\documentclass[12pt]{amsart}

\usepackage{latexsym,amsmath,amsfonts,amsthm,amssymb,color}
\usepackage[utf8]{inputenc}
\usepackage[T1]{fontenc}
\usepackage{csquotes}
\usepackage[english,french]{babel}
\usepackage{soul}
\usepackage[makeroom]{cancel}
\usepackage[hidelinks,hyperfootnotes=false]{hyperref}
\usepackage{bigints} 
\usepackage{enumitem}
\usepackage{tikz}
\usepackage[capposition=top]{floatrow}
\usepackage{pgf,tikz}
\usetikzlibrary{arrows}
\definecolor{ccqqqq}{rgb}{0.8,0,0}
\definecolor{qqttcc}{rgb}{0,0.2,0.8}
\definecolor{cqcqcq}{rgb}{0.75,0.75,0.75}

\usepackage{xcolor}

\newcommand{\blue}[1]{{\color{blue}#1}}

\colorlet{symbols}{black} 
\tikzset{
	dot/.style={circle,fill=symbols,draw=symbols,inner sep=0pt,minimum size=0.4pt},
	basic/.style={draw=symbols},
	>=stealth,
	}

\newtheorem{theorem}{Theorem}
\newtheorem{lemma}[theorem]{Lemma}
\newtheorem{proposition}[theorem]{Proposition}

\newtheorem{remark}[theorem]{Remark}

\newcommand{\R}{{\mathbb R }}

\newcommand{\Z}{{\mathbb Z }}

\renewcommand{\Z}{{\mathbb Z}}

\newcommand{\xx}{\langle x \rangle}
\newcommand{\Ai}{\operatorname{Ai}}
\newcommand{\NL}{\mathrm{NL}}
\renewcommand{\epsilon}{\varepsilon}
\newcommand{\lo}{\mathrm{lo}}
\newcommand{\hi}{\mathrm{hi}}
\newcommand{\lin}{\mathrm{lin}}
\newcommand{\eel}{\mathrm{ell}}

\DeclareMathOperator{\sech}{\mathrm{sech}} 

\newcommand{\newsection}[1]{\section{#1}\setcounter{theorem}{0}
 \setcounter{equation}{0}\par\noindent}

\begin{document}
\selectlanguage{english}
\title{Dispersive decay of small data solutions for the KdV equation 
\\
D\'ecroissance dispersive des solutions \'a donn\'ees petites pour l'\'equation de KdV
}

\author{Mihaela Ifrim}
\address{Department of Mathematics, University of Wisconsin, Madison}
\email{ifrim@wisc.edu}

\author{Herbert Koch}
\address{Mathematisches Institut   \\ Universit\"at Bonn }
\email{koch@math.uni-bonn.de}

\author{ Daniel Tataru}
\address {Department of Mathematics \\
  University of California, Berkeley} 
\email{tataru@math.berkeley.edu}

\begin{abstract}
  We consider the  Korteweg--de Vries (KdV) equation, and prove that small localized data yields solutions which have dispersive decay on a quartic time-scale. This result is optimal, in view of the emergence of solitons at quartic time, as predicted by inverse scattering theory.
  
  \selectlanguage{french}
\noindent
{\sc R\'{e}sum\'{e}.}
Dans cet article nous consid\'erons l'\'equation de Korteweg--de Vries (KdV), et montrons que pour des
donn\'ees petites et localis\'ees, les solutions ont une dynamique
dispersive sur une \'echelle  de temps quartique. Ce r\'esultat
est optimal, comme le pr\'edit la th\'eorie de la diffusion inverse.
\end{abstract}

\selectlanguage{english}

\maketitle 

\tableofcontents

\section{Introduction}

We consider real solutions for the Korteweg--de Vries equation (KdV) 
\begin{equation}
\label{kdv}
\left\{
\begin{aligned}
& u_t + u_{xxx} - 6 uu_x   = 0\\
& u(0) = u_0,
\end{aligned}
\right.
\end{equation}
on the real line. Assuming that the initial data is  small and localized, we seek to understand the long time dispersive properties of the solution.

This has been a long term goal of research in this direction. 
In particular, one natural question to ask is whether, for localized initial data, the solutions to the nonlinear equation exhibit the same dispersive decay as the solutions to the corresponding linear equation. In general this is not the case globally in time,
 due primarily to two types of nonlinear solutions:
 \begin{enumerate}[label=(\roman*)]
 \item Solitons, which move to the right with constant speed.
 \item Dispersive shocks, where the nonlinearity acts like a transport
 term and pushes the dispersive part of the solution to the left.
\end{enumerate}

\medskip

This paper combines some earlier work and insight gained by the authors when analyzing global or long time  dynamical behaviour of solutions to certain models of dispersive equations. Our long term goal is to understand the \emph{soliton resolution conjecture} for the nonlinear Korteweg--de Vries equation (KdV).  Historically, solitary waves (water waves which do not disperse for a long time and which move at a constant speed without changing their shape) were first observed and reported by John Scott Russell in a shallow canal. He called such a wave   ``a wave of  translation, in a wave tank''. This phenomenon was first explained mathematically by Korteweg and de Vries  in \cite{kdv} in 1895. Solitons represent interesting mathematical objects that influence the long time dynamics of the solutions.

The \emph{soliton resolution conjecture} applies to many nonlinear dispersive equations and asserts, roughly speaking, that any reasonable solution to such equations eventually resolves into a superposition of a dispersive component (which behaves like a solution to the linear equation) plus a number of  ``solitons''. This should only be taken as a guiding principle, as many variations can occur; for instance the number of solitons
could be finite or infinite, while the dispersive part might not truly have linear scattering, but instead some modified scattering behavior.

This conjecture was studied in many different frameworks (i.e.\ for different dispersive equations like for example for the nonlinear Schr\"odinger equation (NLS), see \cite{terry} and references within) and it is known in many perturbative cases in the  setting: when the solution is close to a special solution, such as the vacuum state or a ground state, as well as in defocusing cases, where  no non-trivial bound states or solitons exist. But it is still almost completely open in non-perturbative situations (in which the solution is large and not close to a special solution) which contain at least one bound state.  

\medskip

Turning our attention to solutions to the KdV equation with small
initial data, one can distinguish two stages in  the nonlinear evolution from the perspective of soliton resolution. Initially, one expects the solutions to satisfy linear-like
dispersive bounds. This stage lasts until nonlinear effects 
(i.e.\ solitons and dispersive shocks) begin to emerge. The second stage corresponds to solutions which split into at least two  of the
following components: a linear dispersive part, a dispersive shock, and a soliton.

In this article we aim to describe the first of the two stages
above. To better frame the question, we restate the problem as follows:

\medskip
{\bf Question:} \emph{If $\epsilon \ll 1$  is the initial data size, then what is the  time scale up to which the solution will satisfy linear dispersive decay bounds?}

\medskip

Our main result identifies the quartic time scale $T_\epsilon = \epsilon^{-3}$ as the optimal time scale on which linear dispersive
decay for all localized data of size $\leq \epsilon$. 
The precise statement of the result is provided in Theorem~\ref{t:quartic} below.

We prove this, and also provide some heuristic reasoning, based on inverse scattering, as to why this result is optimal, in other
words that the quartic time scale that marks the earliest 
possible emergence of either solitons or dispersive shocks.
To our knowledge this is the first result that rigorously describes the  dispersive decay of the solutions on a quartic time-scale.

\subsection{The linear KdV flow} 

If one removes the nonlinearity and considers instead the linear Korteweg--de Vries equation
\begin{equation}
\label{kdv-lin}
\left\{
\begin{aligned}
& u_t + u_{xxx} = 0 \\
& u(0) = u_0, 
\end{aligned}
\right.
\end{equation}
then the solutions will exhibit Airy type decay.
To better understand this bound, it is useful to separate the 
domain of evolution $(t,x) \in \R^+ \times \R$ into three regions (see Figure 1 above/below):

\begin{enumerate}

\item The hyperbolic region 
\[
H := \{ x \lesssim  -t^{\frac13}\},
\]
where one  sees an oscillatory, Airy type behavior for the solution, with dispersive decay.

\item The self-similar region 
\[
S := \{ |x| \lesssim  t^{\frac13}\},
\]
where  the solution essentially looks like a bump function with $t^{-\frac13}$ decay.

\item The elliptic region, 
\[
E := \{ x \gtrsim  t^{\frac13}\},
\]
which is eventually  left by each oscillatory component of the solution, and consequently we have better decay.
\end{enumerate}

\vspace*{.4cm}
\begin{figure}[ht]
\begin{center}
\begin{tikzpicture}[line cap=round,line join=round,>=triangle 45,x=1.5cm,y=1.5cm]
\draw[->,color=black] (-3,0) -- (3,0)  node[right] {$x$};
\draw[color=red, dashed] (-2,1) -- (2,1)  node[right] {$t=1$};

\draw[->,color=black] (0,0) -- (0,3) node[anchor=north west] {$t$};
\clip(-3,0) rectangle (3,3);
\draw[color=qqttcc,dashed,samples=100,domain=0.0:3.0] plot(\x,{(\x)*(\x)*(\x)});
\draw[color=qqttcc,dashed,samples=100,domain=-3.0:0.0] plot(\x,{ abs((\x)*(\x)*(\x))});
\draw [color=qqttcc](2,2) node[anchor=west] {$E$};
\draw [color=black](1.4,3) node[anchor=north west] {$\vert x\vert =t^{\frac{1}{3}}$};
\draw [color=qqttcc](-2,2) node[anchor=east] {$H$};
\draw [color=qqttcc](.5,2.3) node[anchor=north] {$S$};
\end{tikzpicture}
\end{center}
\floatfoot{Figure 1. The partition into the three main regions: $H$, $E$, and $S$.}
\end{figure}

Consistent with the above partition we define the expression $\xx$ 
 in a time dependent fashion as 
\begin{equation} \label{eq:japanese} 
\xx := (x^2+|t|^\frac23)^\frac12 .
\end{equation} 
Then the following result describes the dispersive decay of linear KdV waves:

\begin{proposition}\label{p:lin}
Assume that the initial data $u_0$ for \eqref{kdv-lin} satisfies 
\begin{equation}\label{lin-data}
\| u_0\|_{H^1} + \|x^2 u_0\|_{L^2} \leq \epsilon.
\end{equation}
 Then the corresponding solution satisfies the bound
\begin{equation}\label{lin-decay}
t^\frac14 \xx^{\frac14} |u(t,x)| + t^\frac34 \xx^{-\frac14} |u_x(t,x)| \lesssim \epsilon.
\end{equation}
Furthermore, in the elliptic region $E$ we have the better bound
\begin{equation}
\label{lin-decay-e}
 \xx |u(t,x)| + t^\frac12 \xx^{\frac12}  |u_x(t,x)|\lesssim \epsilon \ln (t^{-\frac13} \xx).
\end{equation}

\end{proposition}

Here the  $\epsilon$ factor is not important, we have only added it for easier comparison with the nonlinear problem later on.

We also remark that the norm in \eqref{lin-data} is stronger than we 
need. In Section~\ref{s:lin}, where the proposition is proved,
we will in effect state and prove a sharper version, with the same conclusion but 
a weaker hypothesis. Incidentally, the bound \eqref{lin-decay-e} in the 
elliptic region is the one that follows from that weaker hypothesis, and 
can be improved under the assumption \eqref{lin-decay}; we leave the details for the 
interested reader.

\subsection{The nonlinear problem}

KdV is a completely integrable flow, and admits an infinite number of conservation laws. The first few ones are as follows: 
\[
\begin{split}
E_0 = & \   \int u^2\, dx,
\\
E_1 = & \  \int  u_x^2 +  2 u^3\, dx,
\\
E_2 = & \  \int  u_{xx}^2 + 10 u u_x^2  + 5 u^4\, dx,
\\
E_3 = & \ \int u_{xxx}^2 + 14 u u_{xx}^2 + 70 u^2 u_x^2 + 14 u^5\, dx.
\end{split}
\]

In a Hamiltonian interpretation, these energies generate commuting Hamiltonian flows with the 
Poisson structure  defined by the associated Poisson form  (which is the dual 
or inverse of the symplectic form)
\[ 
\Lambda(u,v) = \int u v_x \, dx.
\]
The first of these flows is the group of translations, and  the second is the KdV flow.

The local well-posedness and eventually the global well-posedness for the KdV equations has received a lot of attention over the last twenty years. To frame the discussion that follows we recall the scaling law for KdV, which  is 
\[
u(x,t) \to \lambda^2 u(\lambda x, \lambda^3 t),
\]
and corresponds to the critical Sobolev space $\dot H^{-\frac32}$.

Without being exhaustive we mention  only a few of the results. We begin with the study of the local  $L^2$ well-posedness of the KdV equations  which 
was proved both on the line and on the circle 
by Bourgain in \cite{bourgain}. Refinements of the ideas developed in \cite{bourgain} were further implemented by Kenig-Ponce-Vega in \cite{kpv}; their 
work extended the Sobolev index of the local well-posedness theory down to $s > -3/4$ in  $H^s(\mathbb{R})$, respectively $s>-1/2$ in  the $H^s(\mathbb{R}/\mathbb{Z})$ case. For the Sobolev indices $s=-3/4$ respectively $s=-1/2$ see the work of Christ-Colliander-Tao \cite{cct}, and Colliander-Keel-Staffilani-Takaoka-Tao \cite{ckstt1, ckstt2, ckstt3}.
Using inverse scattering techniques Kappeler and Topalov \cite{MR2267286} proved that the solution maps can be continuously (and globally in time) extended to $H^{-1} $ in the period case. 

The local well-posedness results were extended globally in time in \cite{ckstt2} with the sole exception of the case $s=-3/4$. This was later independently settled  by Guo~\cite{guo}  and Kishimoto~\cite{kishimoto}.  Very recently it was proved by Killip and Visan \cite{MR3820439} that the KdV flow is globally well-posed in $H^{-1}(\mathbb{R})$. This is a definitive result, as it is known that  below $H^{-1}(\mathbb{R})$ the flow map cannot be continuous (see \cite{molinet}).

An important role in the global results was played by the conservation laws for the KdV evolution. 
In addition to the classical conservation laws we also have conservation laws for $H^s$ norms of the solution for $s \geq -1$, see \cite{MR3400442, KT, MR3820439}. We will rely on these conservation laws  in the work that we will present here. In fact, these conservation laws also played a crucial role in the proof of the global well-posedness result in \cite{visan-kdv}.

\subsection{Solitons and dispersive shocks}

The nonlinear KdV evolution shares some of the features
of the linear evolution, but also exhibits some new 
behaviour patterns. Here we discuss two such patterns:
solitons and dispersive shocks.

\bigskip

\noindent \textbf{1. Solitons.}
As it is well-known, the KdV equation admits soliton solutions,
for instance the state 
\[
Q = -2 \sech^2 x 
\]
is a soliton which moves to the right with speed $4$. We can also translate and rescale it. Its rescales for instance are 
\[
Q_\lambda (x)= \lambda^2 Q(\lambda x), 
\]
which move to the right with speed $4 \lambda^2$.

A given KdV solution may contain one or more solitons.
The KdV equation is integrable so one expects  solitons to interact without changing their shape. 

Solitons within a given KdV solution are in a one-to-one correspondence with the negative eigenvalues of the Schr\"odinger operator 
\[ 
\psi \to - \psi'' + u(t) \psi.  
\]
Precisely, an eigenvalue $-\lambda^2$ corresponds to the soliton $Q_\lambda$ up to a possible shift.

The Schr\"odinger operator may have a negative eigenvalue even if the initial data $u_0$ is a small nice bump function, and a soliton will  emerge in this case. The only question is after how long does it happen?  Properties of the discrete spectrum are collected in Proposition \ref{spectrum} in the appendix. In particular we apply the results by Schuur \cite{schuur} to the case when $\phi_0$ is a Schwartz function, $0<\varepsilon$ is small and $u_0= \varepsilon \phi_0$. In this case there is exactly one negative eigenvalue of size $-\varepsilon^2/2$. Hence there is exactly one soliton, which has width $\varepsilon^{-1}$.
 This follows from estimates of Schuur \cite{schuur}.
 Heuristically, one expects this soliton to emerge 
 from the self-similar region when the spatial scales are matched. But this happens exactly at quartic time $\varepsilon^{-3}$.

\bigskip

\noindent \textbf{2. Dispersive shocks.}
If one neglects the third order derivative in the KdV equation
then what is left is the Burgers equation, which develops shocks
in finite time. The third order derivative adds dispersion to the mix,
sending the high frequencies to the left as a dispersive tail. 
This guarantees that shocks as a jump discontinuity cannot form.
Are we still left with a tangible Burgers like effect at low frequencies?
This does indeed happen, and is what we call a \emph{dispersive shock}.

To understand the mechanics of its possible  appearance,
suppose for a moment that the solution for KdV has the same behavior 
as the linear KdV solution in the self-similar region
(we focus our attention there because no oscillations are present). There
the solution has size $ u \approx \epsilon t^{-\frac13}$ and frequency $t^{-\frac13}$. 
On the other hand, interpreting 
the nonlinear term as a transport term, we see that it would shift 
the solution by $\epsilon t^{\frac23}$ within a  dyadic time region. 
This is consistent with the size of the self-similar region only 
if $\epsilon t^{\frac23} \lesssim  t^{\frac13}$, 
or equivalently  $ t \lesssim \epsilon ^{-3}$. Thus for larger times
one cannot expect a linear decay, and instead most of the 
mass will be pushed (if $u$ is positive) into the dispersive region; this of course depends
on the sign of the transport velocity, and would be effective only provided
that $u > 0$ in the self-similar region. 

To understand the shock quantitatively, one can consider another class
of special solutions to the KdV equation, namely the self-similar
solutions. These must be functions of the form
\[
u(t,x) = t^{-\frac23} \phi(x/t^\frac13),
\]
where $\phi$ solves the following Painlev\'e type equation
\begin{equation}\label{painlevetype}  
\frac13 (2 \phi + y \phi_y)  -  \phi_{yyy} + 6 \phi \phi_y  = 0.
\end{equation} 
This admits a one parameter family of solutions, given by the Miura map applied to solutions to the Painlev\'e II equation,
\begin{equation}\label{painleveII}  
  \psi'' - 2 \psi^3 -x \psi = \alpha   
\end{equation} 
with $\alpha \in \R$. If $\psi$ satisfies \eqref{painleveII} then $\phi= \psi'+\psi^2$ satisfies \eqref{painlevetype}.

Of particular interest is a family of solutions to \eqref{painleveII}, parametrized by $-1\le \sigma \le 1$. For such $\sigma$ there exists a unique bounded global solution (see \cite{NIST}) to \eqref{painleveII} which behaves like $\sigma \Ai(x)$ as $ x\to \infty$. It leads to solutions to \eqref{painlevetype}  which decay to the right,
and are oscillatory, Airy type to the left, and are positive around $y=0$ if $|\sigma| < 1$.  We expect these solutions to become important for understanding the large time behaviour near the self-similar region.  The solution with $\sigma=1$
is the famous Hastings-McLeod  solution \cite{hastings}.
In this way  we obtain KdV solutions with $t^{-\frac23}$ decay in the self-similar region. One expects the dispersive shock to cause either 
convergence in the self-similar region to one of these self-similar solutions,
or alternatively, to generate a slow motion along this family.
The first case happens for the Miura map of solutions to mKdV studied by \cite{MR3462131}. Unfortunately this class of solutions is non-generic, see Schuur \cite{schuur} and Ablowitz and Segur \cite{ablowitz}.

\subsection{The main result}

We now turn our attention to the nonlinear KdV equation \eqref{kdv}
with localized data of small size $\epsilon$. For this problem we seek the answer
to the following:
\smallskip

\textbf{ Question:} \emph{What is the optimal time scale, depending on $\epsilon$,
where nonlinear effects can become dominant ?}
\smallskip

As a quantitative version of the above question, we will
ask what is the optimal time scale on which  the linear dispersive decay bounds in Proposition~\ref{p:lin} hold for the nonlinear problem for small decaying initial data.  
Our main result asserts that this timescale is the quartic time scale,
$T_\epsilon = \epsilon^{-3}$:
\begin{theorem}\label{t:quartic}
Assume that the initial data $u_0$ for KdV satisfies 
\begin{equation}\label{small}
\|u_0\|_{\dot B^{-\frac12}_{2,\infty}} + \| x u_0\|_{\dot H^\frac12} \leq \epsilon \ll 1.
\end{equation}
Then for the quartic lifespan
\begin{equation}
|t| \ll \epsilon^{-3},
\end{equation}
we have the dispersive bounds (using the notation  \eqref{eq:japanese})  \begin{equation}
|u(t,x)| \lesssim \epsilon t^{-\frac14} \xx^{-\frac14}  
\qquad |u_x(t,x)| \lesssim \epsilon t^{-\frac34} \xx^\frac14.
\end{equation}
Furthermore, in the elliptic region $E$ we have the better bound
\begin{equation}
 \xx |u(t,x)| + t^\frac12 \xx^{\frac12}  |u_x(t,x)|\lesssim \epsilon \ln(t^{-1/3}\xx).
\end{equation}
The implicit constants are independent of $\varepsilon$ and $u$.
\end{theorem}

A small multiple of the Dirac measure at $0$ is a particular case of initial data satisfying the assumptions. More generally, if the initial data is a Dirac measure, then, with a scaling argument, the theorem gives bounds for the corresponding solution up to a small time.

The time scale in this result is optimal. To clarify this assertion, in the
last section we  discuss the possible emergence of solitons from the dispersive flow at quartic time. In a similar manner, dispersive shocks
may also arise at the same time, as argued in their brief heuristic discussion above.

We also comment on the choice of the norms in the theorem.
The Besov space $\dot B^{-\frac12}_{2,\infty}$ is the minimal one
at low frequency where we can place our initial data:  A smooth localized bump function with nonzero integral is in 
$\dot B^{-\frac12}_{2,q}$ if and only if $q=\infty$. Here only frequencies larger than $\epsilon$
are interesting, and below that we can freely flatten off the $\dot H^{-\frac12}$ Fourier weight. Choosing this norm also at high frequency
is harmless, and allows us to use scaling in order to streamline the analysis.

The $\dot H^\frac12$ norm for $x u$ scales in exactly the same way 
as the above Besov norm. Their combination is  exactly consistent with the pointwise decay rates above even for the linear KdV flow with fully localized data. At a technical level, the $\dot H^\frac12$ corresponds by duality 
to the $\dot H^{-\frac12}$ well-posedness for the linearized equation.
In turn, $\dot H^{-\frac12}$ is an optimal space where this well-posedness for the linearized equation
can be studied in that  well-posedness  in any other Sobolev space $\dot H^s$ implies\footnote{To make this accurate one needs to consider simultaneously the forward and backward well-posedness, as these are interchanged
by duality.} $\dot H^{-\frac12}$ well-posedness.

It is instructive to relate the theorem to inverse scattering techniques. The assumptions we make on the initial data are not strong enough to exclude an infinite number of negative eigenvalues for the corresponding Schr\"odinger operator. In fact, it is not hard to construct a potential $u$ satisfying the assumptions for a given $\varepsilon>0$ with an infinite number of negative eigenvalues. There are a number of papers on  asymptotics for fast decaying initial data \cite{ablowitz, schuur, deift}. 
To our knowledge no quantitative bounds near the self-similar region are available - the difficulty is the emergence of solitons. On the other hand slightly sharpened  asymptotics of Schuur (based on the inverse scattering procedure)   show that solitons emerge at the quartic time scale for a large class of initial data, see the discussion in the appendix.

We remark that the similar problem for the Benjamin-Ono equation was considered in recent work by the first and the last author \cite{IT-BO}. There the 
optimal time scale turns out to be the almost global one, $T_\epsilon = e^{\frac{c}{\epsilon}}$.

The structure of the paper is as follows: in Section~\ref{s:lin} we prove the linear KdV bound in Proposition~\ref{p:lin}. Along the way we introduce some tools which 
will be very useful in the nonlinear analysis later on. In Section~\ref{s:nonlin} we begin the proof of our main result,
and reduce it to four key elements: (i) energy estimates for $u$,
(ii) energy estimates for the linearized equation, (iii) energy estimates for a nonlinear
version $L^{\NL} u$ of $Lu$ related to the  scaling derivative of $u$, and (iv) a nonlinear Klainerman-Sobolev inequality which yields the pointwise estimates starting from the $L^2$
bounds. These four largely independent steps are carried out in the 
following three sections. 

Finally, in the last section we discuss the optimality of our result
in two steps. First we use the inverse scattering tools to discuss the 
possible emergence of solitons from small initial data. Then we provide a heuristic argument for the appearance of dispersive shocks at the 
quartic time scale.

\subsection{A few notations and definitions}
We recall the definition of the scaled Japanese bracket \eqref{eq:japanese} .  
Throughout the paper we use a standard Littlewood-Paley decomposition,
\[
u=\sum _{\lambda }P_{\lambda} u:=\sum_{\lambda} u_{\lambda},
\]
where $\lambda \in 2^\Z$ and $u_{\lambda }$ are frequency localized in dyadic annuli $ \left\{ \vert \xi \vert \approx \lambda\right\} $. We also use the related notations $P_{>\lambda}$,  $P_{<\lambda}$, and correspondingly $u_{>\lambda}$, $u_{<\lambda}$. 

In particular we will use 
the time dependent multipliers $P^+$, $P^-$ and $P_{\lo}$
which select the regions $\{ \xi > t^{-\frac13}\}$, $\{\xi <- t^{-\frac13}\}$ and $\{| \xi |\lesssim t^{-\frac13}\}$. 

With these notations the homogeneous Besov norm $
\dot{B}^{-\frac{1}{2}}_{2,\infty} $ is defined by
\[
\Vert u\Vert_{ \dot{B}^{-\frac{1}{2}}_{2,\infty}} =\sup_{\lambda} \lambda^{-\frac{1}{2}} \Vert u_{\lambda}\Vert_{L^2}.
\]
The homogeneous Sobolev space   $\dot H^\frac12$, $\dot H^{-\frac12}$ are the standard spaces defined by the usual Fourier multipliers.  

\subsection*{Acknowledgments} Mihaela Ifrim was partially supported by a Clare Boothe Luce Professorship. Herbert Koch was partially supported by the Deutsche Forschungsgemeinschaft (DFG, German Research Foundation) through the Hausdorff Center for Mathematics under Germany's Excellence Strategy - EXC-2047/1 - 390685813 and through CRC 1060 - project number 211504053. Daniel Tataru was partially supported by the NSF grant DMS-1800294 as well as by a Simons Investigator grant from the Simons Foundation.

\newsection{Linear analysis}
\label{s:lin}

In this section we consider dispersive bounds for the linear KdV equation
\eqref{kdv-lin}, and prove Proposition~\ref{p:lin}. We begin with a heuristic discussion.

The fundamental solution for \eqref{kdv-lin} can be described using the Airy function,
\[
K(t,x) = t^{-\frac13} \Ai(x/t^\frac13). 
\]
Explicitly, the solution to \eqref{kdv-lin} is 
\[
u(t,x)=K(t,x)\ast u_{0}(x).
\] 
Based on the known asymptotics for the Airy function,
it follows that solutions with integrable localized initial data
\begin{equation}\label{L1loc-data}
\| u_0 \|_{L^1} \leq 1, 
\end{equation}
 and  with the  support supp $u_0$ included in the interval $\left[-1,0\right]$, then the solution and its derivative will satisfy the same decay bounds for $t \gtrsim 1$:
\begin{equation}\label{Airy-decay}
|u(t,x)| \lesssim t^{-\frac14}  \xx^{-\frac14} e^{-\frac{2}{3} x_+^{\frac{3}{2}} t^{-\frac{1}{2}}} 
\mbox{ and }|u_x(t,x)| \lesssim  t^{-\frac34} \xx^\frac14 e^{-\frac{2}{3} x_+^{\frac{3}{2}}t^{-\frac{1}{2}}}.
\end{equation}

Our goal now is to relax the compact support assumption to a decay 
estimate, while, at the same time, to provide a more robust proof
of the pointwise decay bound which will be later adapted to the nonlinear
problem.

Precisely we introduce the time dependent operator
\[
L(t) := x - 3t \partial_x^2,
\]
which is the push forward of the operator $x$ along the linear KdV flow
and which satisfies the following properties:
\[
\left[ \partial_x , L\right]=1,\quad \left[ \partial_t+\partial_x^3, L\right] =0.
\]

If $u$ solves the equation~\eqref{kdv-lin} then so does $Lu$,
therefore we have at our disposal $L^2$ type bounds for both $u$ 
and $Lu$. One might be tempted to try to work with both $u$ and $Lu$ 
in $L^2$. However, it turns out to be more efficient to work in
the following functional framework:
\[
u \in \dot B^{-\frac12}_{2,\infty}, \qquad Lu \in \dot H^{\frac12}.
\]
At time $t = 0$, these norms can be readily estimated in terms of 
the norms in Proposition~\ref{p:lin},
\[
\| u_0 \|_{\dot B^{-\frac12}_{2,\infty}} + \|x u_0 \|_{\dot H^{\frac12}}
\lesssim \| u_0\|_{H^1} + \| x^2 u_0\|_{L^2}.
\]
Because of this, we can replace Proposition~\ref{p:lin}
with the following stronger form:

\begin{proposition} \label{p:lin+}
Assume that the initial data $u_0$ for \eqref{kdv-lin} satisfies 
\begin{equation}\label{lin-data+}
\| u_0 \|_{\dot B^{-\frac12}_{2,\infty}} + \|x u_0 \|_{ \dot H^{\frac12}} \leq 1. 
\end{equation}
 Then the corresponding solution $u$ satisfies the pointwise bounds 
\begin{equation}\label{lin-decay+}
t^\frac14 \xx^{\frac14} |u(t,x)| + t^\frac34 \xx^{-\frac14} |u_x(t,x)| \lesssim 1, \qquad x \in \R.
\end{equation}
Furthermore, in the elliptic region $E$ we have the better bound
\begin{equation}
\xx |u(t,x)| + t^\frac12 \xx^{\frac12}  |u_x(t,x)|\lesssim  \ln (t^{-\frac13} \xx)  .
\end{equation}
\end{proposition}

Furthermore, since both Sobolev norms for $u$ and $Lu$ are 
preserved in time, it will suffice to prove the following
fixed time result:

\begin{lemma}
Let $t >0$. Assume that a function $u\in \dot B^{-\frac12}_{2,\infty}(\R)$ satisfies
\begin{equation}\label{lin-L2}
 \| u \|_{\dot B^{-\frac12}_{2,\infty}} + \| L(t) u\|_{\dot H^{\frac12}} \leq 1 .
\end{equation}
Then it also satisfies the pointwise bounds 
\begin{equation}\label{lin-decay++}
t^\frac14 \xx^{\frac14} |u| + t^\frac34 \xx^{-\frac14} |u_x| \lesssim 1,
\qquad x \in \R.
\end{equation}
Furthermore, in the elliptic region $E$ we have the better bound
\begin{equation}
 \xx |u(x)| + t^\frac12 \xx^{\frac12}  |u_x(x)|\lesssim  \ln (t^{-\frac13} \xx)  .
\end{equation}
\end{lemma}

There are two motivations for using these particular Sobolev norms.
One is linear, and is the fact that with this choice of spaces
the estimates in the above proposition and lemma are invariant with respect to scaling.

A second motivation will come from the nonlinear problem later on, and 
arises from the fact that, while all Sobolev norms are equally good 
for linear energy estimates, this is no longer the case for the nonlinear problem. There, it seems that the $\dot H^\frac12$ norm for the 
nonlinear counterpart $L^{\NL} u$ of $Lu$ is the only one we have access to.

\begin{proof}
We first take advantage of the observation that our bounds in the lemma are 
invariant with respect to scaling in order to rescale the problem and set $t = 1$  and we omit $t$ in the notation.
This will not make a major difference, but simplify the computations somewhat.

We will split the real line into the self-similar region $S$, which after scaling is $S = \{ |x| \lesssim 1 \}$ (would be $ = \{ |x| \lesssim t^\frac13 \}$ in general),
the elliptic region $E = \{ x \gg 1 \}$ and the hyperbolic region $H = \{ -x \gg 1\}$. Furthermore, we split the last two regions into dyadic components.
We begin with some elliptic $L^2$ bounds in dyadic regions
$A_R = \{ \xx \approx R \gtrsim 1\}$. By a slight abuse we also denote 
$A_1 = \{ \xx \lesssim 1\}$.
To address some of the issues arising from our 
use of the $\dot B^{-\frac12}_{2,\infty}$ and $\dot H^{\frac12}$ norms, we 
start our analysis with some elliptic bounds. 

\medskip
\textbf{A. A low frequency bound.} The $\dot H^\frac12$ bound for $Lu$ does not see 
the constants in $Lu$. More quantitatively, in a dyadic region $A_R$ functions at frequencies 
below $1/R$ are indistinguishable from constants. In order to be able to localize our estimates
to dyadic scales, it is essential to be able to better estimate these low frequencies in $Lu$.
Precisely, we prove that
\begin{lemma} \label{l:Lu-ell}
Assume that \eqref{lin-L2} holds. Then
\begin{equation}\label{Lu-lo-lin}
\| L u\|_{L^2(A_R)} \lesssim R^\frac12.
\end{equation}
\end{lemma}

\begin{proof} To prove this, we split $u$ at the frequency cut-off $R^{-1}$,
\[
u = u_{< R^{-1}} + u_{\gtrsim R^{-1}}.
\]
Then use the Besov bound on $u$ to compute
\[
\|  L  u_{<R^{-1}} \|_{L^2(A_R)} \lesssim R \| u_{<R^{-1}} \|_{L^2(\mathbb{R})} \lesssim R^\frac12.
\]
On the other hand
\[
L  u_{\gtrsim  R^{-1}} = P_{\gtrsim R_{-1}} Lu + [L, P_{\gtrsim R^{-1}}] u =  P_{\gtrsim R_{-1}} Lu + [x,P_{\gtrsim R^{-1}}] u,
\]
and the conclusion follows since the commutator is a Fourier multiplier of size $R$ supported near $|\xi| \sim R^{-1}$.
\end{proof}

\medskip

\textbf{B. A high frequency bound.} The balance of the two terms in $Lu$ indicates that 
the bulk of $u$ in $A_R$ is localized at frequency $R^\frac12$. In the next lemma
we take advantage of this balance in order to improve the regularity of $u$ in $A_R$ at 
high frequency $ > R^{\frac12}$:
\begin{lemma} \label{l:u-ell}
Assume that \eqref{lin-L2} holds. Then
\begin{equation}\label{u-hi-lin}
\|u\|_{L^2(A_R)} \lesssim R^{\frac14}, \qquad
\|u_x\|_{L^2(A_R)} \lesssim R^\frac34, \qquad
\|u_{xx}\|_{L^2(A_R)} \lesssim R^\frac54.
\end{equation}
\end{lemma}
\begin{proof} The bound for the low frequencies of $u$ (i.e.\ below $R^\frac12$) follows directly from the Besov bound in \eqref{lin-L2}, irrespective of the spatial localization.
Hence it suffices to consider the high frequencies of $u$, $\lambda \gg R^\frac12$. For these we have 
\begin{equation}\label{eq:lul}
L  u_\lambda = P_{\lambda} Lu + \lambda^{-1} u_\lambda ,
\end{equation}
where, by a slight abuse of notation, the $u_{\lambda}$ on the right stands for
a generic frequency $\lambda$ unit multiplier applied to $u$.
Thus, using again \eqref{lin-L2},  we compute
\begin{equation}\label{lul-ul}
\| L u_\lambda \|_{L^2} \lesssim \lambda^{-\frac12}, 
\qquad
\| u_\lambda \|_{L^2} \lesssim \lambda^\frac12 . 
\end{equation}

To obtain the bound on the derivative, we integrate by parts in $A_R$ to get
\[
 \int_{\mathbb{R}} \chi_R |u_{\lambda, x}|^2 \, dx= \int_{\mathbb{R}} \frac{1}{3}\chi_R u_\lambda Lu_\lambda + 
\left (\frac{1}{2} \chi''_R -\frac{1}{3} x \chi_R \right ) |u_\lambda|^2\, dx,
\]
which yields the preliminary bound
\[
\| u_{\lambda,x} \|_{L^2(A_R)}^2 \lesssim  R \lambda.
\]
We now express $u_\lambda$ in terms of $u_{\lambda,x}$, and localize,
\[
\chi_R u_\lambda = \chi_R \partial^{-1}_{x,\lambda} u_{\lambda,x}
= \partial^{-1}_{x, \lambda} \left( \chi_R u_{\lambda,x}\right)  - [\partial^{-1}_{x, \lambda}, \chi_R]  u_{\lambda, x},
\]
where the antiderivative $\partial^{-1}_{x, \lambda}$ is localized at  frequency $\lambda$.
The integral kernel decays polynomially away from the diagonal, which suffices to add up the contributions from the areas $A_{R'}$. This yields a local bound for $u_\lambda$,
\[
\|u_{\lambda} \|_{L^2(A_R)} \lesssim R^\frac12 \lambda^{-\frac12}.
\]

Repeating the argument above we then have 
\begin{equation}
\|u_{\lambda,x}\|_{L^2(A_R)} \lesssim  R \lambda^{-\frac12},
\end{equation}
and further 
\begin{equation}
\|u_{\lambda}\|_{L^2(A_R)} \lesssim  R \lambda^{-\frac32}.
\end{equation}
By \eqref{eq:lul} and \eqref{lul-ul} we can easily obtain the last bound 
\[
\|u_{\lambda,xx}\|_{L^2(A_R)} \lesssim \lambda^{-\frac12} +  R^2 \lambda^{-\frac32},
\]

and hence the proof is complete.
\end{proof}

\medskip

\textbf{C. Localization}.
Here we use the elliptic bounds in \textbf{A}, \textbf{B} to conclude that we can
localize the problem to the region $A_R$ simply by
replacing $u$ by $v := \chi_R u$. All the norms here are restricted to the region $A_R$. Hence, we will for example, write $L^2$ instead of $L^2(A_R)$, just for the sake of simplicity. We will use this notation throughout this section (i.e.\ in paragraphs \textbf{D}, \textbf{E}, and \textbf{F} within this section).

 Here $v$ solves an equation of the form
\[
(x - 3 \partial_x^2 ) v = f 
\]
in $A_R$, where we control 
\begin{equation}\label{v-gen}
\|v\|_{L^2(A_R)} \lesssim R^{\frac14}, \qquad
\|v_x\|_{L^2(A_R)} \lesssim R^\frac34, \qquad \|v_{xx}\|_{L^2(A_R)} \lesssim R^\frac54.
\end{equation}
and, with $f$ supported in $A_R$, 
\begin{equation}\label{f-gen}
\| f \|_{\dot H^\frac12 \cap R^\frac12 L^2} \lesssim 1.
\end{equation}
\medskip 

\textbf{D. Pointwise estimate in the hyperbolic region.}
Here we consider the region $A_R^H$ to the left of the origin, and use 
hyperbolic energy estimates to establish the desired pointwise 
bound for $v$ supported in $A_R^H$. 

Here the primary frequency is $\lambda = R^\frac12$,
but $f$ is worse at lower frequency than at higher frequencies, and we need to account for this.
For expository purposes assume at first that this is not the case, 
i.e.\ that $f$ simply satisfies the low frequency bound
\[
\| f\|_{L^2} \lesssim R^{-\frac14}.
\]
Then we simply treat the $v$ equation as a hyperbolic evolution equation
and use an energy estimate,
\[
\frac{d}{dx} \left( -x |v|^2 + 3 |v_x|^2\right) = -|v|^2 - 2 f v_x, 
\]
and then apply Gronwall's inequality on the $R$ dyadic region
to obtain the pointwise bound
\[
\sup_{x \in A_R^H} \left(-x|v|^2  + 3 |v_x|^2\right) \lesssim \|v_x\|_{L^2} \|f\|_{L^2}\lesssim R^{\frac{3}{4}}\cdot R^{-\frac{1}{4}}  \lesssim R^\frac12,
\]
which suffices.

Consider now the situation in ~\eqref{f-gen}, where a direct estimate of $f v_x$ 
would yield logarithmic losses in the dyadic frequency summation. To avoid those
we use the frequency scale $R^\frac12$ to split
\[
f = f_{\lo} + f_{\hi},
\]
and correspondingly
\[
f v_x = f_{\lo} v_x + f_{\hi} v_x = - f_{lo,x} v + f_{\hi} v_x + \partial_x( f_{\lo} v),
\]
where $f_{\lo}:=\tilde{\chi}_{R}f_{< R^{\frac{1}{2}}}$, and $f_{\hi}:=\tilde{\chi}_{R}f_{\geq R^{\frac12}}$. Here $\tilde{\chi}_R$ is also a characteristic function similar to $\chi_R$ but with a larger support than $\chi_R$.

 Now we view the last term as an energy correction,
\[
\frac{d}{dx} \left(-x |v|^2 + 3 |v_x|^2  -2 f_{\lo} v\right) = -|v|^2 - 2 f_{\hi} v_x - 2 f_{lo,x} v,
\]
and using Gronwall's inequality again we obtain
\[
\sup_{x \in A_R^H} \left(-x |v|^2 + 3 |v_x|^2\right)  \lesssim R^\frac12 + \sup  |f_{\lo} v|.
\]
For $f_{\lo}$ we get from \eqref{f-gen} by Bernstein's inequality 
\[
\| f_{\lo} \|_{L^\infty} \lesssim (\ln R)^\frac12,
\]
where the $\ln$ loss arises from the dyadic summation in the frequency range 
\[ 
\{ R^{-1} \leq |\xi| \leq R^\frac12\}.
\]
 This again leads to the desired bound
\[
\sup_{x\in A^r_h} \left( -x |v|^2 + 3 |v_x|^2\right)  \lesssim R^\frac12.
\]

\medskip 

\textbf{E. Pointwise estimate in the self-similar region.}

This follows from \eqref{v-gen} and Sobolev embeddings.

\medskip 

\textbf{F. Pointwise estimate in the elliptic region.}

Here we split again $f = f_{\lo} + f_{\hi}$.
The leading part of $v$ will then be $x^{-1} f_{\lo}$.
Subtracting that, we are left with
\[
v_1 := v - x^{-1} f_{\lo},
\]
which solves
\[
L v_1 = f_1:= f_{\hi} + 3\partial_x^2 (x^{-1} f_{\lo}).
\]
Here we can easily estimate $f_1$ using \eqref{f-gen},
\[
\| f_1 \|_{L^2} \lesssim R^{-\frac14}.
\]
This allows us to estimate integrating by parts in the following identity
\[
\begin{aligned}
\int_{\mathbb{R}}v_1Lv_1\, dx =\int_{\mathbb{R}} f_1v_1\, dx 
\end{aligned}\]
and arrive at
\[
\int_{\mathbb{R}}x\vert v_1\vert ^2\, dx +3\int_{\mathbb{R}}\vert v_{1,x}\vert ^2\, dx =\int_{\mathbb{R}}f_1v_1\, dx.
\]
Using Cauchy-Schwartz inequality implies 
\[
R\Vert v_1\Vert_{L^2}^2 +3\Vert v_{1,x}\Vert^2_{L^2}\lesssim \Vert f_1\Vert _{L^2}\Vert v_1\Vert_{L^2} ,
\]
which further leads to
\begin{equation}
\label{bounds needed}
R\Vert v_1\Vert_{L^2}^2 +3\Vert v_{1,x}\Vert^2_{L^2}\lesssim R^{-1}\Vert f_1\Vert^2 _{L^2}.
\end{equation}
Thus, using the bound on $f_1$, we arrive at 
\[
\| v_1\|_{L^2} \lesssim R^{-\frac54}, \qquad 
\| \partial_x v_1 \|_{L^2} \lesssim R^{-\frac34},
\]
and further using the $v_1$ equation,
\[
\| \partial_x^2 v_1 \|_{L^2} \lesssim R^{-\frac14}.
\]
Now we can obtain pointwise bounds for $v_1$
by Sobolev embeddings,
\[
|v_1| \lesssim R^{-1}, \qquad |\partial_x v_1| \lesssim R^{-\frac12}. 
\]
This is exactly as needed. On the other hand for the 
$x^{-1}f_{\lo}$ we proceed as we did before, and we use Bernstein's inequality, in order to obtain a 
similar bound but with a log loss.
\end{proof}

\newsection{The nonlinear quartic result}
\label{s:nonlin}

In this section we describe the main building blocks in the proof
of Theorem~\ref{t:quartic}, and show how these can be used to conclude 
the proof of Theorem~\ref{t:quartic}.

The proof of the result is based on energy estimates. The difficulty
is that we need to take full advantage of the nonresonant structure
of the equation. Primarily, in our setting we expect resonant
interactions to primarily occur in the self-similar region $\{ |x|
\lesssim t^{\frac13} \}$, which corresponds to frequencies $\lesssim t^{-\frac13}$.

Following the pattern in the linear analysis in Section~\ref{s:lin}, one of our 
energy estimates will be for $u$. The second energy estimate in the linear case is for $Lu$. Unfortunately, in the nonlinear case $Lu$ no longer solves a good equation, so we will seek a nonlinear replacement for it $L^{\NL} u$.
In view of the scaling symmetry, one solution for the linearized equation 
\begin{equation}\label{linearized}
z_t + z_{xxx} = 6  \partial_x ( u z)
\end{equation}
is provided  by the function
\[
z = \partial_x ( x u - 3 t u_{xx} + 9 t u^2) + u.
\]
However, given our initial data assumption and the linear estimates in Section~\ref{s:lin} we would rather like to work at the
level of $\partial^{-1} z$. If $z$ solves \eqref{linearized} then $w:=\partial^{-1} z$ formally solves
 the adjoint linearized equation
\begin{equation}\label{adj-linearized}
w_t + w_{xxx} = 6( u w_x).
\end{equation} 

However, working with $\partial^{-1} u$ does not seem like a good idea 
unless we assume that the function has zero average, i.e.\ that following equality holds
\[
\displaystyle\int_{\mathbb{R}} u\, dx  = 0.
\]
Because of that, we will work instead with the function
\[
w = L^{\NL} u : = x u - 3 t u_{xx} + 9 t u^2.
\]
This in turn solves an inhomogeneous adjoint linearized equation
\begin{equation}\label{LNL-eq}
w_t + w_{xxx} = 6( u w_x) + 3 u^2.
\end{equation}

To start with, we recall the bounds we seek to prove, namely 
\begin{equation}\label{point-re}
|u(t,x)| \lesssim \epsilon t^{-\frac14} \xx^{-\frac14},
\qquad |u_x(t,x)| \lesssim \epsilon t^{-\frac34} \xx^\frac14.
\end{equation}
Our proof will be a nonlinear version of the linear argument
in Section~\ref{s:lin}, but organized as a bootstrap argument. 

We will work with solutions in a time interval $[0,T]$, where $T$ 
will be chosen later. Our main bootstrap assumption will be
\begin{equation}\label{point-boot}
|u(t,x)| \leq M \epsilon t^{-\frac14} \xx^{-\frac14},  
\qquad |u_x(t,x)| \leq M \epsilon t^{-\frac34} \xx^\frac14 , \qquad
t \in [0,T],
\end{equation}
where $M \gg 1$ is a large universal constant also to be chosen later.
We will use this in order to both prove the desired conclusions
and to improve the bootstrap bounds \eqref{point-boot}. For this to work
the constant $M$ is chosen first, and then $T$ is chosen small enough depending on $M$, 
\begin{equation} \label{T-choice}
T \ll_M \epsilon^{-3}.    
\end{equation}

Given this set-up, our proof has four main steps:

\bigskip

\textbf{I. Uniform energy estimates for $u$.}
Here no bootstrap assumption is necessary, and the 
main bound is translation invariant. To motivate the norms we 
will use, we start with the homogeneous Besov space $\dot B^{-\frac12}_{2,\infty}$ which is the best we can do for the initial data
at low frequency. Tracking the time evolution of this homogeneous norm seems difficult
at low frequency, so instead we will seek to replace it with an inhomogeneous norm below a well chosen threshold frequency.

To motivate the choice of the threshold frequency, we start by observing that up to time $t$, frequencies below $t^{-\frac13}$ in $u$ do not have any interesting linear KdV dynamics. Because of that, it seems 
wasteful to use the homogeneous Besov norm below this scale. In our case,
the quartic lifespan corresponds to $t \leq \epsilon^{-3}$, so the above frequency threshold is exactly  $\epsilon$. Based on that, we define the inhomogeneous Besov space
$B^{-\frac12,\epsilon}_{2,\infty}$, 
where we make the norm inhomogeneous below frequency $\epsilon$,
\[
B^{-\frac12,\epsilon}_{2,\infty} :=  \dot B^{-\frac12}_{2,\infty} +  \epsilon^{\frac12} L^2,
\]
or equivalently
\[
\| u \|_{B^{-\frac12,\epsilon}_{2,\infty}} = \inf_{u = u_1+ \epsilon^\frac12 u_2} \| u_1\|_{\dot B^{-\frac12}_{2,\infty}}
+ \| u_2\|_{L^2},
\]
where the two components are matched exactly at frequency $\epsilon^\frac12$.

Then our uniform energy estimate is as follows:

\begin{proposition}\label{p:energy}
Assume the solution $u$ to KdV equation \eqref{kdv} has initial data $u_0$ so that
\[
\| u_0\|_{ B^{-\frac12,\epsilon}_{2,\infty}} \leq \epsilon .
\]
Then 
\begin{equation}\label{energy}
 \sup_{t\in \R} \|u(t)\|_{ B^{-\frac12,\epsilon}_{2,\infty}}  \lesssim  \epsilon.
\end{equation}
\end{proposition}
This result is derived in Section~\ref{s:energy} from the 
$H^{-1}$ conservation law for KdV obtained in \cite{KT} (see also 
the earlier bounds in \cite{MR3400442} and the bounds
in \cite{MR3820439}).

\bigskip

\textbf{II. $\dot H^{-\frac12}$ bounds for the linearized equation.}
The main result in this step is as follows:
\begin{proposition}\label{p:linearized}
Let $u$ be a solution to the KdV equation \eqref{kdv} in a time interval $[0,T]$
which satisfies the smallness assumption \eqref{small} 
for the initial data, as well as the bootstrap assumption
\eqref{point-boot}. 
Assume that $T$ is as in \eqref{T-choice}.
Then the linearized equation \eqref{linearized} is well-posed in $\dot H^{-\frac12}$
with uniform bounds
\begin{equation}\label{linest12}
\|w(t) \|_{\dot H^{-\frac12}} \approx \|w(0)\|_{\dot H^{-\frac12}}, 
 \qquad  t \in [0,T].
\end{equation}
\end{proposition}

Equivalently, the adjoint linearized equation \eqref{adj-linearized} is well-posed 
in $\dot H^\frac12$ with uniform bounds. 
We note here that the implicit constant in \eqref{linest12} does not depend on the bootstrap constant $M$. Instead, $M$ appears only in the choice of the quartic time constant.

This result is proved in Section~\ref{s:linearized}, and 
serves as a key tool in the next step.
\bigskip

\textbf{III. Uniform  $\dot H^\frac12$ bounds for $L^{\NL} u$.} We recall that 
$L^{\NL} u$ solves the equation \eqref{LNL-eq}, which is the adjoint linearized equation with an $u^2$ source term. In view of the result in step II, it is thus 
natural to seek estimates for $L^{\NL} u$ in the space $\dot H^\frac12$.
 Using the linear estimates above, this amounts to proving appropriate bounds for the $u^2$ inhomogeneity.   We will show the following:

\begin{proposition}\label{p:LNL}
Let $u$ be a solution to the KdV equation \eqref{kdv} in a time interval $[0,T]$, which satisfies the smallness assumption \eqref{small} 
for the initial data, as well as the bootstrap assumption
\eqref{point-boot}. Assume that $T$ is as in \eqref{T-choice}.
Then we have
\begin{equation}\label{LNLest}
\|L^{\NL} u(t) \|_{\dot H^\frac12} \lesssim \epsilon, 
\qquad t \ll_M \epsilon^{-3} \qquad  t \in [0,T].
\end{equation}
\end{proposition}

This result is proved in Section~\ref{s:LNL}, and will play the same role  as the similar bound for $Lu$ in the linear case.

\bigskip

\textbf{IV. Nonlinear Klainerman-Sobolev inequalities.}
Taking into account the uniform bounds for $u$ in 
Proposition~\ref{p:energy} and for $L^{\NL} u$ in Proposition~\ref{p:LNL},
the desired pointwise bound \eqref{point-re}
will follow from the following:
\begin{proposition}\label{p:KS}
Assume that $T$ is as in \eqref{T-choice} and $0\le t \le T$. 
Let $u(t)$ be a  a function which satisfies the bootstrap assumption
\eqref{point-boot} associated to the time $t$, as well as the $L^2$ bounds

\begin{equation}\label{e-est}
 \| u{(t)} \|_{\dot B^{-\frac12}_{2,\infty}} + \| L^{\NL} u{(t)} \|_{\dot H^{\frac12}}
 \lesssim \epsilon.
\end{equation}

\noindent a) Then we have the pointwise bound 
\begin{equation}\label{KS13}
t^\frac14 \xx^{\frac14} |u(t,x)| + t^\frac34 \xx^{-\frac14} |u_x(t,x)| 
 \lesssim \epsilon.
\end{equation}
\noindent b) In the elliptic region $E$ we have the additional bound 
\begin{equation}\label{KS-E}
\xx |u(t,x)| + t^\frac12 \xx^{\frac12} |u_x(t,x)| 
 \lesssim \epsilon \ln (t^{-1/3} \xx).
\end{equation}
\end{proposition}
This result is derived in Section~\ref{s:KS}. We point out 
that the bounds in this proposition are fixed time bounds, i.e.
they only involve $u(t)$ and make no reference to the KdV equation;
nevertheless the time $t$ is still present in the statement, 
as both the hypothesis and the conclusion depend on $t$.
\bigskip

One sees that at the conclusion of steps I-III above
we obtain the bound \eqref{e-est} in the time interval $[0,T]$
provided that $T \ll_M \epsilon^{-3}$.
 Here it is crucial that the constant $M$ 
in the bootstrap assumption~\eqref{point-boot} does not 
influence the implicit constant in \eqref{e-est}.
Then applying step IV above we obtain the desired pointwise
bounds \eqref{point-re}, again with implicit constants 
independent of $M$, in the same range $0 \leq t  \leq T \ll_M \epsilon^{-3}$. 

Thus we can first choose  $M$ to be a sufficiently large universal constant,
so that \eqref{KS-E} improves \eqref{point-boot}.  Then we choose
$T$ small enough (depending on $M$) as in \eqref{T-choice}.
This concludes the bootstrap argument, since it is obvious that the dependence of the implicit constants on $M$ is monotone,  provided that $M$ is a sufficiently large universal constant.

\bigskip

\newsection{Energy estimates} \label{s:energy}

The goal of this section  is to establish the uniform bounds for $u$ in Proposition~\ref{p:energy}, which involve the Besov space $B^{\blue{-}\frac12,\epsilon}_{2,\infty}$. 
This is an  easy consequence of uniform $H^{-1}$ bounds in \cite{MR3400442} as well as of the $H^{-1}$ energy functional constructed in \cite{KT}, and 
a special case of Theorem 1.2   by Killip-Visan-Zhang 
\cite{MR3820439}:  

\begin{theorem}
There exists $\delta > 0$ and an energy functional 
\[
E^{-1}: B_\delta(H^{-1}) := \left\{ u\in H^{-1}\, , \, \Vert u\Vert_{{H^{-1}}} \leq \delta \right\} \to \R^+,
\]
so that 

(i) Norm equivalence: $E^{-1}$ is equivalent to the $H^{-1}$ norm,
\[
E^{-1}(u) \approx \| u\|_{H^{-1}}^2.
\]

(ii) Conservation: $E^{-1}$ is conserved along the KdV flow. 
\end{theorem}

\begin{proof}[Proof of Proposition~\ref{p:energy}]
We interpret the (nearly) homogeneous Besov norm in the proposition in terms 
of $H^{-1}$ norms by using the rescaled KdV solution
\[
u^{[\alpha]}(x,t) = \alpha^{-2} u(x/\alpha,t/\alpha^3),
\]
 where the frequency $\lambda$ for $u$ corresponds to the frequency $1$ for $u^{[\lambda]}$. 
 Using this scaling applied with $\lambda \geq \epsilon$, the Besov norm in $B^{-\frac12,\epsilon}_{2,\infty}$ can be expressed 
as 
\[
\|u\|_{B^{-\frac12,\epsilon}_{2,\infty}} \approx \sup_{\lambda \geq \epsilon} \lambda \| u^{[\lambda]} \|_{H^{-1}}.
\]
Here the norm on the right essentially selects the $\lambda$
frequencies of $u$. The $\lambda$ factor arises because the KdV 
scaling is at the $\dot H^{-\frac32}$ level, whereas here we are measuring 
$\dot H^{-\frac12}$ type norms.  Precisely, going in one direction we have for  $\lambda \geq \epsilon $
\[
\| P_\lambda u\|_{B^{-\frac12,\epsilon}_{2,\infty}}
\approx \lambda^\frac12 \| P_\lambda u\|_{\dot H^{-1}} 
\approx \lambda \|P_{1} u^{[\lambda]}\|_{\dot H^{-1}} \lesssim \lambda \| u^{[\lambda]}\|_{H^{-1}},
\]
where $P_\lambda$ and $P_1$ are standard dyadic Littlewood-Paley projectors.
The other direction is similar.

At the initial time $t = 0$ the norm on the left 
has size $\ll \epsilon$, so all the norms on the right have size 
$\ll 1$. Hence the above theorem applies, and they (i.e.\ the norms) are 
approximatively conserved. This yields the desired bound for the Besov norm.
\end{proof}

\newsection{Bounds for the linearized equation}
\label{s:linearized}

The aim of this section is to prove Proposition~\ref{p:linearized}. Throughout the section
we will assume that $u$ solves the KdV equation and satisfies the uniform energy estimates 
\eqref{energy} in Proposition~\ref{p:energy} as well as our bootstrap assumptions \eqref{point-boot}. 

Using the standard notation $D=-i \partial_x$, we switch to a new variable 
\[
y := |D|^{-\frac12} w,
\]
which solves the equation
\begin{equation}
\label{y-eq}
(\partial_t + \partial_x^3) y = 6 H |D|^\frac12 ( u |D|^\frac12 y).
\end{equation} 
where $H= \frac{D}{|D|}$ is the usual Hilbert transform.

This new variable has the role to shift our problem in an $L^2$ setting, and also to simplify the exposition of the paper. Thus, for this equation we need to prove uniform bounds for the  $L^2$ norm of $y$,
\[
E^{[2]} (y) = \|y\|_{L^2}^2,
\]
namely
\[
E^{[2]} (y(t)) \approx E^{(2)}(y(0)). 
\]

We have 
\begin{equation} \label{eq:timeder} 
\partial_t E^{[2]} (y)  = - 12 \int  H |D|^\frac12 y \cdot u |D|^\frac12 y \, dx.
\end{equation} 
The expression on the right is too large to be estimated directly in terms of $\|y\|_{L^2}$.
However, it is nonresonant when all three entering frequencies are nonzero, so 
we can try to eliminate it using a normal form energy correction. Precisely, we will seek to eliminate 
(the bulk of) this expression by adding a cubic 
correction to the quadratic energy functional, at the 
expense of producing further quartic errors; these quartic errors will be bounded.

In this paragraph we will explain the heuristics which are meant to justify the energy correction we will consider below. Thus we begin with our initial KdV equation \eqref{kdv} for which we can formally compute the normal form transformation that removes the quadratic nonresonant terms:
\[
\tilde u = u - (\partial^{-1}_x u)^2.
\]
Here $\tilde{u}$ is the normal form variable which will satisfy a KdV like-equation: the linear part of the equation we obtain after implementing the normal form transformation is the same as in \eqref{kdv}, but there are no quadratic terms, only cubic ones.
However this is singular at frequency $0$.  Nevertheless this issue can be bypassed if we truncate  in a self-similar fashion, avoiding the low frequencies on the scale $|\xi| \lesssim t^{-\frac13}$,
\[
\tilde{\tilde{u}}: = u - (\partial^{-1}_x u_{\geq t^{-\frac13}})^2,
\]
and thus making the normal form rigorous.  We now go further and compute the normal form transformation for the linearized equation (which is the linearization of the original normal form) \eqref{linearized}, which at the formal level is given by
\[
\tilde{w}=w-2\partial^{-1}_xu\cdot \partial^{-1}_xw.
\]
The same truncation as above will also fix the singularity issue encountered at frequency zero. However, we are interested in correcting the functional energy corresponding to the $y$ equation \eqref{y-eq}. For this equation we also have a normal form transformation as the quadratic terms are nonresonant, and based on the definition of $y$ and its connection with the the linearized equation \eqref{kdv-lin}, the normal form (formal expression) is given by
\[
\tilde{y} = y  +2 |D|^{-1/2}  \partial^{-1}_x u\cdot H|D|^{-\frac{1}{2}} y.
\]
To determine what the cubic correction should be, we go ahead and proceed as in \cite{BH}. 
Hence, formally, the correction to the energy would be 
\begin{equation}
\label{correction}
E^{(3)}= 4 \int   H |D|^{-\frac12} y \cdot \partial^{-1}_x  u |D|^{-\frac12} y \, dx. 
\end{equation}

We have two issues here: \textbf{i)} we do not know apriori that this correction (i.e.\ $E^{(3)}$) is bounded, but we will show this is the remaining part of this section; \textbf{ii)}  \eqref{correction} cannot be used as it is because of low frequency issues. To remedy this,
we will estimate directly all the low frequency contributions to $\partial_t E^{[2]} (y)$ (see \eqref{eq:timeder}), choosing
the self-similar frequency scale $t^{-\frac13}$ as the truncation threshold.
We apply the standard  Littlewood-Paley trichotomy, which 
asserts that the two highest frequencies must be comparable while the third may be smaller.
Because of this, there are three cases to consider:

\medskip

(i) Three low frequencies:
\[
 \left| \int  H |D|^\frac12 y_{\lesssim  t^{-\frac13}} \cdot  u_{\lesssim t^{-\frac13}} \cdot
 |D|^\frac12 y_{\lesssim t^{-\frac13}}  \, dx \right|\lesssim  t^{-\frac13}  \|u\|_{L^\infty} 
\|y \|_{L^2}^2 \lesssim \epsilon M t^{-\frac23}  \|y\|_{L^2}^2 .
\]
Here we  use Cauchy-Schwartz inequality together with the bootstrap assumption \eqref{point-boot}.

\medskip

(ii) Low frequency on $u$. Here we have
\[
\int  H |D|^\frac12 y_{\gg t^{-\frac13}}\cdot u_{ \lesssim t^{-\frac13}}  \cdot |D|^\frac12 y_{\gg  t^{-\frac13}} \, dx = 0,
\]
as $H$ is skew-adjoint and we can commute it across $u$.

\medskip

(iii) Low frequency on either $y$ factor. Here we move the fractional derivative to the product of the two other factors and    compute using   a fractional Leibniz rule:
\[
\begin{aligned}
 \left| \int  H |D|^\frac12 y_{\gg t^{-\frac13}} \cdot  u_{\gg t^{-\frac13}} \cdot 
 |D|^\frac12 y_{\lesssim t^{-\frac13}}  \, dx \right| 
 &\lesssim  t^{-\frac16}  \||D|^\frac12   u_{\gg t^{-\frac13}}  \|_{L^\infty} 
 \|y \|_{L^2} \| y_{\lesssim t^{-\frac13}}\|_{L^2}\\
&\lesssim M \epsilon t^{-\frac23}  \|y\|_{L^2}^2 .
\end{aligned}
\]
Here we also get a milder commutator term when switching the half-derivative onto $u$. The pointwise bound on $|D|^{\frac{1}{2}}u$ follows as an interpolation of the bootstrap bounds in \eqref{point-boot}.

\medskip
To summarize, we have proved 
\begin{equation}\label{dtE2}
\partial_t E^{[2]} (y)  = 12  \int  H |D|^\frac12 y^{\hi}  \cdot u^{\hi} \cdot |D|^\frac12 y^{\hi} \, dx + O(M \epsilon t^{-\frac23}) \| y\|_{L^2}^2,
\end{equation}
where
\[
y^{\hi} := y_{\gg t^{-\frac13}}, \qquad u^{\hi} := u_{ \gg t^{-\frac13}}.
\]
This simplification allows us to use a restricted normal form energy correction,
\begin{equation}\label{E3}
E^{[3]}:= 4 \int   H |D|^{-\frac12} y^{\hi} \cdot \partial^{-1}_x  u^{\hi} \cdot |D|^{-\frac12} y^{\hi} \, dx,
\end{equation}
which  is a  rigurous truncation at high frequencies of the functional $E^{(3)}$ defined in \eqref{correction}. 
Then we define the modified energy as 
\[
E := E^{[2]} + E^{[3]},
\]
and we need to prove norm equivalence,
\begin{equation}\label{E-eq}
E(y) \approx \|y\|_{L^2}^2,
\end{equation}
and slow growth,
 \begin{equation}\label{E-grow}
\partial_t E(y) \lesssim M \epsilon t^{-\frac23}  \|y\|_{L^2}^2.
\end{equation}

For the first bound we estimate 
\[
|E^{[3]}(y)| \lesssim t^{\frac13} \| \partial^{-1}_x u^{\hi } \|_{L^\infty} \|y\|_{L^2}^2  \lesssim \epsilon t^{\frac13}\|y\|_{L^2}^2,
\]
which suffices on the quartic time scale.

It remains to prove \eqref{E-grow}. For that, using also \eqref{dtE2}, we compute
\[
\partial_t E(y)  =  D_1 + D_2 + D_3 + O(M \epsilon t^{-\frac23})\|y\|_{L^2}^2,
\]
where all cubic terms arising from $\partial_x^3$ cancel because of our choice of the correction:

\bigskip

(i) $D_1$ arises from the scale change in the truncation,
as the multiplier $P^{\hi}$ is time dependent, with symbol
of the form
\[
p^{\hi}(\xi) := \chi(t^\frac13 \xi).
\]
Its time derivative has the form
\[
\partial_t p^{\hi}(\xi) = t^{-1} \frac{t^\frac13 \xi}3  \chi'(t^\frac13 \xi),
\]
which is supported exactly in the region $|\xi| \approx t^{-\frac13}$, and we harmlessly abbreviate it as
\[
t^{-1} P_{t^{-\frac13}}.
\]
Then the corresponding error term is 
\[
\begin{aligned}
\! D_1 = &  \! \int \!  t^{-1} H |D|^{-\frac12} y^{\hi} \cdot  \partial^{-1}_xu_{t^{-\frac13}} \cdot |D|^{-\frac12} y^{\hi} \, dx
+ \! \int \!  t^{-1} H |D|^{-\frac12} y_{t^{-\frac13}} \cdot  \partial^{-1}_x u^{\hi} \cdot  |D|^{-\frac12} y^{\hi} \, dx \! 
\\
&+  \int  t^{-1} H |D|^{-\frac12} y^{\hi} \cdot  \partial^{-1}_x u^{\hi} \cdot  |D|^{-\frac12} y_{t^{-\frac13}} \, dx.
\end{aligned}
\]

\bigskip

(ii) $D_2$ is the quartic term arising from $u_t$, 
\[
D_2 =  \int   H |D|^{-\frac12} y^{\hi}\cdot    (u^2)^{\hi} \cdot |D|^{-\frac12} y^{\hi} \, dx.
\]

\bigskip

(iii) $D_3$ is the quartic term arising from $y_t$,
\[
D_3 =  \int   ( u |D|^\frac12 y)^{\hi} \cdot  \partial^{-1}_x u^{\hi} \cdot  |D|^{-\frac12} y^{\hi} \,dx.
\]

For $D_1$ we use the pointwise bounds
\begin{equation}
\label{hi-point-bound}
|\partial^{-1}_x u^{\hi}| + |\partial^{-1}_x u_{t^{-\frac13}}| \lesssim M \epsilon
\end{equation}
to compute
\[
|D_1| \lesssim M \epsilon t^{-\frac23} \|y\|_{L^2}^2.
\]

For $D_2$ we use the pointwise bound on $u$ to estimate
\[
|D_2| \lesssim t^{\frac13} \|u^2\|_{L^\infty} \|y\|_{L^2}^2 \lesssim M^2 \epsilon^2 t^{-\frac13} \|y\|_{L^2}^2,
\]
which again suffices.

Finally for $D_3$ we write
\[
\begin{split}
D_3 = & \ \int |D|^\frac12 y\cdot u\, \cdot \,  P^{\hi}( \partial^{-1}_x u^{\hi} \cdot  |D|^{-\frac12} y^{\hi})  \, dx 
\\
= & \  \int y \cdot |D|^\frac12 \left[ u \cdot  P^{\hi}\left( \partial^{-1}_x u^{\hi} \cdot |D|^{-\frac12} y^{\hi} \right)\right] \, dx .
\end{split}
\] 
Then we distribute $|D|^\frac12$ to each of the other factors using a fractional Leibniz rule
to get 
\[
\begin{aligned}
|D_3| \lesssim  \| y\|_{L^2} & \left(  \| |D|^\frac12 u\|_{L^\infty} \| \partial^{-1}_x u^{\hi}\|_{L^\infty} 
\||D|^{-\frac12} y^{\hi}\|_{L^2}+ \right. \\ & \left.\  \| u\|_{L^\infty} \| |D|^{-\frac12} u^{\hi}\|_{L^\infty} \||D|^{-\frac12} y^{\hi}\|_{L^2} 
+  \| u\|_{L^\infty} \| \partial^{-1}_x u^{\hi}\|_{L^\infty} \| y^{\hi}\|_{L^2} \right).
\end{aligned}
\]
We again use the bootstrap bounds \eqref{point-boot} and and the high frequency bounds \eqref{hi-point-bound} we conclude that
\[
\begin{aligned}
|D_3|\lesssim & \ ( M\epsilon t^{-\frac12}\cdot M\epsilon \cdot  t^{\frac16} + M\epsilon t^{-\frac13} \cdot M\epsilon t^{-\frac16} 
\cdot t^{\frac16} + M\epsilon t^{-\frac13} \cdot M\epsilon)   \|y\|_{L^2}^2 
\\
\lesssim & \ M^2 \epsilon^2 t^{-\frac13}    \|y\|_{L^2}^2 .
\end{aligned}
\]
Thus \eqref{E-grow} is proved, and the conclusion of Proposition \ref{p:linearized}  follows via a direct application 
of Gronwall's inequality for the modified energy functional $E(y)$.

\newsection{\texorpdfstring{$\dot H^\frac12$}{} bounds for \texorpdfstring{$L^{\NL} u$}{}.}
\label{s:LNL}
Our aim here is to prove the $\dot H^\frac12$ bound for $L^{\NL} u$ in Proposition~\ref{p:LNL}. 
Here we assume that $u$ is a solution to the KdV equation \eqref{kdv}, which satisfies the uniform energy bounds
given by Proposition~\ref{p:energy}, as well as the pointwise bootstrap assumptions in 
\eqref{point-boot}.

To improve the clarity of the proof, we will add to this a  second bootstrap assumption, namely
\begin{equation}\label{L-boot}
\| L^{\NL} u \|_{\dot H^\frac12} \leq M \epsilon,
\end{equation}
where $M$ is the same as in \eqref{point-boot}.
In order to streamline various computations we will make the 
harmless additional assumption 
\[
M^2 \epsilon T^\frac13 < 1.
\]

We recall that $w = L^{\NL} u$ solves the inhomogeneous adjoint  linearized equation \eqref{LNL-eq},
which we recall here:
\begin{equation}\label{LNL-eq-re}
w_t + w_{xxx} = 6 ( u w_x) + 3u^2.
\end{equation}
By Proposition~\ref{p:linearized} we know that this equation is well-posed in $\dot H^{\frac12}$
with uniform bounds. In order conclude the proof of Proposition~\ref{p:LNL}
we need to have a good way to deal with the source term $u^2$.

One might at first hope that this term can be treated perturbatively,
i.e.\ estimated directly in $L^1_t \dot H^{\frac12}$. This 
indeed turns out to be the case within the self-similar region.
The elliptic region is also favourable due to the better decay, but
the hyperbolic region is a problem due to the weaker Airy decay for $u$. However, the redeeming feature there turns out to be that the 
bilinear interaction in $u^2$ is largely nonresonant, and can be 
treated using a normal form type correction. To implement the above heuristics we will prove the following:
\begin{proposition} \label{p:u2-dec}
Assume that $u$ solves the KdV equation and satisfies the energy bounds \eqref{energy}
and the bootstrap assumptions \eqref{point-boot} and \eqref{L-boot}.
Then the function $u^2$ admits the representation
\begin{equation}\label{u2-dec}
u^2 = P_{\lin} w_1 + f_1,
\end{equation}
where  $P_{\lin}$ refers to the linear part of \eqref{LNL-eq-re} and 
the functions $w_1$ and $f_1$ satisfy the uniform $\dot H^\frac12$ bounds
\begin{equation}
\| w_1(t)\|_{\dot H^\frac12}  \lesssim\epsilon  M^2  (\epsilon t^\frac13),
\end{equation}
respectively
\begin{equation}
\| f_1(t)\|_{\dot H^\frac12} \lesssim \epsilon t^{-1}  M^2  (\epsilon t^\frac13).
\end{equation}
\end{proposition}

It is easily seen that, given this proposition, the conclusion of Proposition~\ref{p:LNL} follows 
easily by applying Proposition~\ref{p:linearized} to $w = L^{\NL} u - w_1$. 
The remainder of this section is devoted to the proof of the above proposition. 
Along the way, we will establish some additional bounds on $u$ and $L^{\NL} u$,
which will also be useful in the proof of the nonlinear Klainerman-Sobolev inequalities 
in the next section.

\subsection{ The decomposition of $u^2$}\label{s:dec-u2}
To define the functions $w_1$ and $f_1$ above  we begin with a linear decomposition
of $u$, using the spectral projectors (multipliers) $P_{\lo}$ and $P^{\pm}$
defined based on the time dependent $t^{-\frac13}$ threshold by 
\[
P_{\lo} := P_{< t^{-\frac13}}, \qquad P^{\pm} := P^\pm_{\geq t^{-\frac13}},
\]
so that 
\[
1 = P_{\lo} + P^{+}+P^- .
\]
This produces a corresponding decomposition of $u$, namely
\[
u = u_{\lo} + u^+ + u^-, \qquad u_{\lo} := P_{\lo} u, \qquad u^{\pm}:= P^\pm u.
\]
We note that $u_{\lo}$ is real, whereas $u^{\pm}$ are complex conjugate of each
other. 

We split $u^2$ into 
\[
u^2 = (u^+)^2 + (u^{-})^2 + f_{2} + f_3,
\]
where 
\[
f_2 := - u_{\lo}^2 + 2 u \cdot u_{\lo}, \qquad f_3 := 2u^+\cdot u^- .
\]
Here we expect $u_{\lo}$ to have better decay at infinity, so we will place $f_{2}$ into $f_1$.
The product in $f_3$ does not have better decay but instead is localized close to frequency zero,
so its $\dot H^\frac12$ norm will be better; thus we will also place it in $f_1$.

The remaining two terms are large, but have the redeeming feature that their interaction 
is nonresonant. Hence for them we will apply the normal form analysis. This will yield the quadratic correction 
\[
w_1 := \partial^{-1} _x\left( (\partial^{-1}_x u^+)^2 + (\partial^{-1}_x u^-)^2\right).
\]
This is chosen so that the quadratic terms in $P_{\lin} w_1$
give exactly $(u^+)^2 + (u^{-})^2$. However, $P_{\lin} w_1$ will also have cubic terms, as both the equation \eqref{kdv} for $u$ and the linearized equation have quadratic terms. Hence we obtain
a relation of the form
\[
P_{\lin} w_1 =( u^+)^2 + (u^{-})^2 + f_4 + f_5
\]
where the cubic terms $f_4$ and $f_5$ are as follows:

- $f_4$ arises from the quadratic term in the KdV equation,
\[
f_4 :=6 \partial^{-1}_x \left( \partial^{-1}_x u^+ P^+(u^2) + 6\partial^{-1}_x u^-  P^-(u^2) \right),
\]

-$f_5$ arises from the quadratic term in $P_{\lin}$,
\[
f_5 :=6 u \cdot \partial_x w_{1}.
\]
These we will seek to place in the perturbative box $f_1$.
Thus we will set 
\[
f_1 = f_2 + f_3 + f_4 + f_5.
\]
Now that we have the decomposition \eqref{u2-dec} for $u^2$, it remains to prove the desired estimates.
We remark that from here on, all the estimates in this section are at fixed time.

\subsection{ Elliptic bounds for $u$ and $L^{\NL} u$ }
\label{ss:ell}
As a preliminary step to estimating the functions $u^{\pm}$
and $u_{\lo}$, we need to improve our understanding of $u$
and $L^{\NL} u$.
For that,  we have to repeat the elliptic estimates in Lemmas~\ref{l:Lu-ell},\ref{l:u-ell} in the nonlinear
setting, under the bootstrap assumptions \eqref{point-boot} and \eqref{L-boot}. 

However, we will also need to reuse these elliptic estimates
in slightly greater generality in the proof of the Klainerman-Sobolev inequalities in Section~\ref{s:KS}. Because of this, in this 
subsection we will replace the bootstrap assumption \eqref{L-boot}
with the following variation:
\begin{equation}\label{L-boot+}
\| L^{\NL} u \|_{\dot H^\frac12} \leq M_L \epsilon.
\end{equation}
where $M_L$ is assumed to satisfy
\[
1 \leq M_L \leq M
\]
For the purpose of this section we could simply take $M_L = M$.
However, as the conclusion of the bootstrap argument in this section we will obtain that  
the above bound holds with $M_L = 1$, and then in the proof of the Klainerman-Sobolev inequalities in Section~\ref{s:KS} we will use
instead $M_L = 1$.

The results will be stated in full generality, but for the proofs it will be 
convenient to rescale to $t = 1$.  Here this can be done
using the exact scaling associated to the KdV equation.
Precisely, given the equation
\[
(x - 3t \partial^2_x) u + 3 t u^2 = f, \qquad f:= L^{\NL} u 
\]
we make the substitution
\[
\tilde u(x) := t^\frac23 u (t, xt^\frac13), \qquad \tilde f(x) :=  t^\frac13 f (t, xt^\frac13).
\]
Now $\tilde u$ and $\tilde f$ solve the same equation but with $t = 1$,
\begin{equation}\label{resc-L}
(x - 3 \partial^2_x) \tilde u + 3  \tilde u^2 = \tilde f .
\end{equation}
Our energy bound for $u$ in \eqref{energy} becomes
\begin{equation}\label{resc-u}
\| \tilde u \|_{B^{-\frac12}_{2,\infty}} \lesssim  \tilde \epsilon,
\end{equation}
where the new smallness parameter $\tilde \epsilon$ is given by
\[
\tilde \epsilon := \epsilon t^{\frac13} \ll 1.
\]
On the other hand the bootstrap  bounds \eqref{point-boot} and \eqref{L-boot} for $u$ and $f$ transferred to 
$\tilde u$ and $\tilde f = L^{\NL}_{|t  = 1} \tilde u$ become 
\begin{equation}\label{resc-boot}
| \tilde u(x)| \leq M \tilde \epsilon \xx^{-\frac14}, \qquad |\tilde u_x(x)| \leq M \tilde \epsilon \xx^\frac14,
\end{equation}
respectively 
\begin{equation}\label{resc-f}
\| \tilde f \|_{\dot H^\frac12} \lesssim M_L \tilde \epsilon.
\end{equation}
In this setting we are assuming for simplicity that  
\begin{equation} \label{order_M} 
M \tilde \epsilon \leq M_L \leq M.
\end{equation} 

As in the analysis of the linear equation in Section~\ref{s:lin},
we begin with a low frequency bound for $L^{\NL} u$:
\begin{lemma}\label{l:Lu-lo}
Under the assumptions \eqref{point-boot} and \eqref{L-boot+}
we have 
\begin{equation}\label{Lu-lo}
\| L^{\NL} u\|_{L^2(A_R)} \lesssim M_L \epsilon R^\frac12.
\end{equation}
\end{lemma}
\begin{proof}
As discussed above, by rescaling, we can set $t = 1$. 
As in the proof of Lemma~\ref{l:Lu-ell} we split $u$ at the frequency cutoff $R^{-1}$,
\[
u = u_{< R^{-1}} + u_{> R^{-1}},
\]
and compute
\[
L^{\NL} u = P_{> R^{-1}} L^{\NL} u + P_{< R^{-1}} L u + P_{< R^{-1}} ( u^2).
\]
The first term is estimated by \eqref{L-boot+} and the second by the Besov norm of $u$  as in the linear case in  Lemma~\ref{l:Lu-ell}. 
For the third one we use our bootstrap assumption \eqref{resc-boot} to get 
\[
|u^2| \lesssim \frac{ M^2 \tilde \epsilon^2}{\xx^\frac12},
\]
which yields
\[
\| P_{< R^{-1}} (u^2) \|_{L^2(A_R)} \lesssim M^2 \tilde \epsilon^2 R^{\frac12} \lesssim M_L \tilde \epsilon R^{\frac12} 
\]
as needed.
\end{proof}
\medskip

We now continue with the counterpart of Lemma~\ref{l:u-ell}, namely 
\begin{lemma}\label{l:u-hi}
Under the assumptions \eqref{point-boot} and \eqref{L-boot+}
we have 
\begin{equation}\label{u-hi} \begin{split}
\|u\|_{L^2(A_R)} \lesssim M_L \epsilon t^{-\frac14} R^\frac14, \qquad
\|u_x\|_{L^2(A_R)} \lesssim M_L \epsilon t^{-\frac34} R^\frac34, 
\\
\|u_{xx}\|_{L^2(A_R)} \lesssim M_L
\epsilon t^{-\frac54} R^\frac54. \qquad \qquad \qquad
\end{split}
\end{equation}
\end{lemma}

\begin{proof}
Again we rescale to $t =1 $.
It suffices to consider the high frequencies of $u$, $\lambda \geq R^\frac12$. For these we have 
\[
L u_\lambda = P_{\lambda} L^{\NL} u + [P_\lambda,x ] u - P_\lambda(u^2) .
\]
As before we show that
\[
\| L u_\lambda \|_{L^2} \lesssim \tilde \epsilon R^\frac14 \lambda^{-\frac12}, 
\qquad
\| [P_\lambda, x] u \|_{L^2} \lesssim \tilde \epsilon \lambda^{-\frac12}.
\]
The only difference is that we now also need to estimate the nonlinear term; but for this purpose 
the nonlinear term only plays a perturbative role.
Using  our bootstrap assumption we have
\[
\| P_\lambda (u^2) \|_{L^2(A_R)} \lesssim \tilde \epsilon^2 M^2 ,
\]
and
\[
 \| P_\lambda \partial_x (u^2) \|_{L^2(A_R)} \lesssim \tilde \epsilon^2 M^2 R^\frac12.
\]
Therefore, using \eqref{order_M},  
\[
\|P_\lambda (u^2) \|_{L^2(A_R)} \lesssim \tilde \epsilon^2 M^2 R^{\frac14}\lambda^{-\frac12} \lesssim 
\tilde \epsilon M_L R^{\frac14}\lambda^{-\frac12},
\]
which suffices. Now the argument is completed as in the linear case.
\end{proof}

The bounds above on $u$ and on $L^{\NL} u$ allow us to localize the function $u$
spatially as follows. Given a dyadic $R \geq t^{\frac13}$ we consider a bump function
$\chi_R$ selecting the region $\{ |x| \approx R \}$ if $R > t^{\frac13}$, respectively 
the region $\{ |x| \lesssim R \}$ if $R = t^{\frac13}$. We denote the localization of $u$
by 
\[
u_R := \chi_R u.
\]
Where necessary we will distinguish between the elliptic and hyperbolic regions 
by using the notations $\chi_R^h$ and $ \chi_R^e$, respectively $u_R^h$ and $u_R^e$.
Multiplying by $\chi_R$ in the $L^{\NL} u = f$ equation we obtain an equation for $u_R$, namely
\begin{equation}\label{eq:u_R}
(x - 3t \partial_x^2) u_R + t u u_R = f_R,
\end{equation}
where
\[
f_R = \chi_R f + t\chi'_R u_x + t \chi''_R u.
\]
By  Lemma~\ref{l:Lu-lo} and Lemma~\ref{l:u-hi},   $u_R$ and $f_R$ satisfy the bounds 
\begin{equation}\label{uR-hi}
\|u_R\|_{L^2} \lesssim M_L \epsilon t^{-\frac14} R^\frac14, \quad
\|u_{R,x}\|_{L^2} \lesssim M_L \epsilon t^{-\frac34} R^\frac34, 
\quad \|u_{R,xx}\|_{L^2} \lesssim M_L
\epsilon t^{-\frac54} R^\frac54,
\end{equation}
respectively
\begin{equation}\label{fR-hi}
\|f_R\|_{\dot H^{\frac12} + R^\frac12 L^2} \lesssim M_L \epsilon.
\end{equation}
This localization will be used for the remainder of this section with $M=M_L$, as well as in the proof of the nonlinear Klainerman-Sobolev estimates in Section~\ref{s:KS}, where we use it with $M_L=1$.

\subsection{ Bounds for $u_{\lo}$ and $u^{\pm}$ }

The pointwise bounds for the components of $u$ are the same as those for $u$, namely
\begin{equation}\label{point-boot-pm}
|u_{\lo}|+ |u^\pm| \lesssim M\epsilon t^{-\frac14} \xx^{-\frac14}, \qquad |\partial_x u_{\lo}|+ |\partial_x u^\pm| \lesssim M \epsilon t^{-\frac34} \xx^{\frac14}.
\end{equation}
However, we expect the bulk of $u$ in the hyperbolic region in $x < 0 $ to be concentrated at frequency $(|x|/t)^\frac12$,
so $u_{\lo}$ as well as the low frequency parts of $u^{\pm}$ should be better behaved.
We begin with the pointwise bounds for $u_{\lo}$:

\begin{lemma}
The low frequency part $u_{\lo}$ of $u$ satisfies
\begin{equation}\label{ulo}
|u_{\lo}| \lesssim M \epsilon \xx^{-1} \ln \left( \xx t^{-\frac13}\right).
\end{equation}
\end{lemma}

\begin{proof}
The bound follows from our bootstrap assumption \eqref{point-boot} if $|x| \lesssim t^{\frac13}$.
For larger $x$ we write
\[
x u_{\lo} = [x,P_{\lo}] u + 3 t \partial_x^2 u_{\lo} + P_{\lo}(t u^2 +f),
\]
and estimate pointwise all terms on the right. 

The commutator is $t^{\frac13}$ times a mollifier on the $t^{\frac13}$ scale, so by 
\eqref{point-boot} it satisfies
\[
|[x,P_{\lo}] u| \lesssim \epsilon M t^{\frac13} \frac{1}{t^\frac14 \xx^\frac14}, 
\]
which suffices. The same bound also follows for the second term, as the $x$ derivatives
contribute $t^{-\frac13}$ factors.

For $tu^2$ we also use \eqref{point-boot} to write
\[
|tu^2| \lesssim M^2 \epsilon^2 t^\frac12 \xx^{-\frac12} \lesssim M \epsilon 
t^{\frac16} \xx^{-\frac12},
\]
which survives after localization and is even better.

Finally we consider the contribution of $f$, which we expand as 
\[
P_{\lo} f = \sum_R P_{\lo} (\chi_R f).
\]
For $\chi_R f$ we use the corresponding component of \eqref{fR-hi}. For the 
dyadic components of $\chi_R f$ we use Bernstein's inequality, which 
yields an $M \epsilon$ bound. After dyadic summation in the frequency range $R^{-1} \lesssim \lambda \lesssim t^{-\frac13}$ we obtain the extra logarithmic loss in the Lemma.
\end{proof}

We continue with bounds for the low frequencies of $u^{\pm}$:

\begin{lemma}
The functions $u^\pm$ satisfy
\begin{equation}\label{upm}
| \partial^{-1}_x u^{\pm}| \lesssim  M \epsilon t^\frac14 \xx^{-\frac34} ,
\qquad
| |D|^{-\frac32} u^{\pm}| \lesssim  M \epsilon t^\frac12 \xx^{-1} .
\end{equation}
\end{lemma}

\begin{proof}
Since the multipliers $\partial^{-j} P_{\pm}$ have kernels which are localized on the 
$t^{\frac13}$ spatial scale, it suffices to separately consider the functions
\[
u_R^\pm := \chi_{R} u^{\pm}.
\]
The case $R \lesssim t^{\frac13}$ follows directly from \eqref{point-boot},
so we consider larger $R$. The high frequencies $(\gtrsim (R/t)^\frac12)$ of $u_R$ 
are also estimated directly from \eqref{point-boot}, so we can discard them from 
$u_R$. 

We now consider in greater detail the bound for $\partial^{-1}_x u^\pm$. We write
\[
x \partial^{-1}_x u_{R}^\pm = t \partial_x u_R^\pm + [ x, \partial^{-1}_x P^{\pm}] u_R +  \partial^{-1}_x P^{\pm} ( t u u_R + f).
\]
For the first term we use directly \eqref{point-boot}.  The commutator
$[ x, \partial^{-1}_x P_{\pm}] $ equals $t^{\frac23}$ times an averaging operator 
on the $t^{\frac13}$ scale, so we can also use \eqref{point-boot} to estimate
\[
| [ x, \partial^{-1}_x P_{\pm}] u | \lesssim t^{\frac23} M \epsilon \frac{1}{t^\frac14 \xx^\frac14}
= M \epsilon \frac{t^\frac5{12}}{ \xx^\frac14}.
\]
For the third term we use \eqref{point-boot} twice, while for the last term we use Bernstein's inequality to obtain
\[
|\partial^{-1}_x f| \lesssim M_L \epsilon t^{-\frac13},
\]
which is better than we need. This concludes the proof of the bound for $\partial^{-1}_x u^{\pm}$.
The bound for  $\partial^{-\frac32}_x u^{\pm}$ is entirely similar.
\end{proof}

Finally, we will need 

\begin{lemma}
Assume that \eqref{uR-hi} and \eqref{fR-hi}, 
as well as the bootstrap assumption \eqref{point-boot}
hold at time $t \ll_M \epsilon^{-3}$.
Then in the hyperbolic region $x<0$  we have the pointwise bound
\begin{equation}\label{hyp0-est}
| (\partial_x - i \sqrt{3}|x|^\frac12 t^{-\frac12}) u^+ | \lesssim
M \epsilon \xx^{-\frac12} t^{-\frac12} \ln ( \xx t^{-\frac13}),
\end{equation}
and in the elliptic region $x>0$

\begin{equation}\label{ell-est}
| \partial_x u^+|  +  |x|^\frac12 t^{-\frac12} | u^+ | \lesssim
M \epsilon \xx^{-\frac12} t^{-\frac12} \ln ( \xx t^{-\frac13}).
\end{equation}

\end{lemma}

\begin{proof}
We will prove the compact bound
\begin{equation}\label{hyp-est}
| (\partial_x - i (-3x)^\frac12 t^{-\frac12}) u^+ | \lesssim
M \epsilon \xx^{-\frac12} t^{-\frac12} \ln ( \xx t^{-\frac13}),
\end{equation}
where the expression $(-3x)^\frac12$ selects the positive square root if $x <0$, i.e.\ in the hyperbolic region,
but is allowed to be either imaginary root if $x > 0$, i.e.\ in the elliptic region. This reflects the fact that the pointwise bounds are better there.

As discussed earlier, we can rescale and reduce the problem
to the case $t=1$, in which case the bound on $t$ translates
into ${\tilde \epsilon} \ll_M 1$. 

Arguing as above, we localize to the region $A_R$ and work with $u^+_R$. In doing that we loose the sharp frequency 
localization; instead we only retain an improved 
bound for the negative frequencies,
\begin{equation}\label{no-neg}
    \| P^- u_R^+ \|_{H^N} \lesssim \epsilon.
\end{equation}

If $R \lesssim  1$ then the bound \eqref{hyp-est} follows directly from \eqref{point-boot-pm}. Hence in the sequel we assume that 
$R \gg 1$. Denoting  
\[
v = (\partial_x - i (-3x)^\frac12) u^+_R
\]
we can write an equation for $v$ as follows:
\[
(\partial_x + i (-3x)^\frac12) v = g,
\]
where 
\begin{equation*} 
\begin{split}
g = & \ (\partial_x + i (-3x)^\frac12 )(\partial_x - i (-3x)^\frac12 ) u_R^+ 
\\
= & \ ( x - 3 \partial^2_x)  u_R^+ + \frac{3i}2 (-3 x)^{-\frac12}  u_R^+
\\
= & \  \chi_R P^+ (u^2 + f) +  \chi_R [ x,P^+] u + 
O(R^{-1}) u^+_x + O( R^{-\frac12}) u^+.
\end{split}
\end{equation*} 
Using \eqref{point-boot-pm} for  all the $u$ terms and the low frequency bound \eqref{upm} for the commutator we get
\begin{equation}\label{gdecom} 
g =  \chi_R P^+ f + O( M \epsilon t^{-\frac56} R^{-\frac12})  .
\end{equation}
Now $v$ is essentially localized at positive frequencies whereas the operator 

\[
Q=(\partial_x +  i (-3x)^\frac12 )
\]
has symbol $i(\xi + (-3x)^\frac12)$ which
is elliptic in the larger frequency region
\begin{equation}\label{ell}
\{ \xi > - 1/4 R^{\frac12}\}.
\end{equation}
Thus we can find a microlocal (semiclassical) parametrix
$Q_{+}^{-1}(x,D)$ for it in this region
with the following properties:

\medskip

(i) Symbol bounds
\[
\left|\partial_x^\alpha \partial_\xi^\beta q_+^{-1}(x,\xi)\right| \lesssim R^{-\alpha} (R^\frac12 + |\xi|)^{-\beta -1}.
\]
\medskip

(ii) Approximate inverse at positive frequencies,
\[
P_{\eel} v = P_{\eel}  Q_{+}^{-1}(x,D) Q v + O_{L^2}(R^{-N}),
\]
where $P_{\eel}$ is a multiplier selecting the region
\eqref{ell}.

\medskip
This in particular guarantees that the kernel $K(x,y)$ of $Q_{+}^{-1}(x,D)$
satisfies
\begin{equation}\label{param-kernel}
|K(x,y)|\lesssim (1+R^\frac12|x-y|)^{-N}     
\end{equation}

To estimate $v$ we use directly the bound \eqref{no-neg}
for $(1-P_{\eel}) v$ to get $\epsilon O(R^{-N})$, and similarly for the error term in $P_{\eel} v$ above. 

Then it remains to estimate the remaining expression
$Q_{+}^{-1}(x,D) g$ where $g$ is as in \eqref{gdecom} above. For this 
we distinguish three 
main contributions:
\medskip

a) From $f$ frequencies below $R^\frac12$, by \eqref{param-kernel} we get 
roughly
\[
Q_{+}^{-1}(x,D) f_{<  R^\frac12} \approx  R^{-\frac12}  f_{<  R^\frac12},
\]
where we use Bernstein's inequality and \eqref{fR-hi} loosing a log.

\medskip

b) From $f$ frequencies above 
$R^\frac12$ we get 
roughly
\[
Q_{+}^{-1}(x,D) f_{>  R^\frac12} \approx  \partial_x^{-1}  f_{>  R^\frac12},
\]
which is as above, but without the log loss.

\medskip

c) For the remaining source term, i.e.\ the last term in  \eqref{gdecom}, we use again \eqref{param-kernel} to get 
\[
Q_{+}^{-1}(x,D)   O( M \epsilon t^{-\frac56} R^{-\frac12})  =     O( M \epsilon t^{-\frac56} R^{-1}) 
\]
which is better than needed in \eqref{hyp-est}.
\end{proof}

\subsection{ Proof of Proposition~\ref{p:u2-dec}}
We successively consider the bounds for $w_1$, $f_2,f_3,f_4$ and $f_5$, which were defined in Section~\ref{s:dec-u2}:
\bigskip

\textbf{(i) The bound for $w_1$.} We consider the "$+$" term,
where we need to estimate the $L^2$ norm of
\[
D^{\frac12} w_1^+ := D^{-\frac12} ( \partial^{-1}_x u^+\cdot \partial^{-1}_x u^+ ).
\]
Here the two inner frequencies are both positive; we denote their dyadic sizes 
by $\lambda_1, \lambda_2 \gtrsim t^{-\frac13}$. Then the outer multiplier must have size
$\lambda_{max} = \max \{ \lambda_1,\lambda_2 \}$. After a Littlewood-Paley decomposition
and separating the two frequencies, we obtain a representation
\[
D^{\frac12} w_1^+ = \sum_{t^{-\frac13} \leq \lambda_1 \leq \lambda_2} 
\lambda_2^{-\frac32} \lambda_1^{-1}  u_{\lambda_1}^+ u_{\lambda_2}^+.
\]
Clearly \eqref{point-boot-pm} also holds for $u_\lambda^+$.
Combining this with \eqref{upm} we obtain
\[
| u_\lambda^+| \lesssim M \epsilon \min \{ t^{-\frac14} \xx^{-\frac14},
\lambda^\frac32 t^{\frac12} \xx^{-1}\},
\]
where the two terms balance exactly at $\lambda \approx \xx^\frac12 t^{-\frac12}$. Summing up over $\lambda_1,\lambda_2$ we obtain
\[
\left| |D|^\frac12 w_1 \right| \lesssim M^2 \epsilon^2 
t^\frac34 \xx^{-\frac74},
\]
and
\[
\|w_1\|_{\dot H^\frac12} \lesssim \epsilon^2 
t^\frac34 t^{-\frac{5}{12}} = \epsilon^2 M^2 t^{\frac13} ,
\]
exactly as needed.

\bigskip

\textbf{(ii) The bound for $f_2$.} 
Here we use \eqref{point-boot} and \eqref{ulo} to estimate pointwise
\[
\begin{split}
|u u_{\lo} | \lesssim & \  M \epsilon \xx^{-\frac14} t^{-\frac14} \xx^{-1} 
M \epsilon \ln ( \xx t^{-\frac13})
\\
\lesssim & 
M^2 \epsilon^2  t^{-\frac14}  \xx^{-\frac54}\ln ( \xx t^{-\frac13}) ,
\end{split}
\]
and a similar bound for $\partial_x(u u_{\lo})$ with an added $(\xx/t)^\frac12$ factor.
Hence for the half derivative we obtain
\[
\left||D|^\frac12 (u u_{\lo}) \right| \lesssim 
M^2 \epsilon^2  t^{-\frac12}  \xx^{-1}\ln ( \xx t^{-\frac13}) ,
\]
and we can now bound its $L^2$ norm by
\[
\left\||D|^\frac12 (u u_{\lo}) \right\|_{L^2} \lesssim M^2 \epsilon^2  t^{-\frac23}
\]
as needed.

\bigskip

\textbf{(iii) The bound for $f_3$.} 
Here we will estimate $u^+ u^{-}$ in $\dot H^\frac12$.
We start with the pointwise bound
\[
|u^+ u^{-}| \lesssim M^2 \epsilon^2 t^{-\frac12} \xx^{-\frac12}.
\]
Next we differentiate,
\[
\partial_x (u^+ u^{-}) = (\partial_x - i (-3x)^\frac12 t^{-\frac12}) u^+ u^- 
+ u^+ (\partial_x + i \ (-3x)^\frac12 t^{-\frac12}) u^-.
\]
Then using \eqref{hyp-est} and \eqref{point-boot-pm} we get
\[
| \partial_x (u^+ u^{-}) | \lesssim M\epsilon x^{-\frac12} t^{-\frac12} \ln ( \xx t^{-\frac13})
\cdot  M \epsilon x^{-\frac14} t^{-\frac14}    = \epsilon^2 x^{-\frac34} t^{-\frac34} \ln ( \xx t^{-\frac13}),
\]
and interpolating,
\[
\| u^+ u^- \|_{\dot H^\frac12} \lesssim  M^2 \epsilon^2 t^{-\frac23},
\]
which suffices.

\bigskip

\textbf{(iv) The bound for $f_4$.} 
We use the pointwise bounds \eqref{upm} for $\partial^{-1} u^\pm$ and 
\eqref{point-boot} for $u$ to obtain 
\[
|D^\frac12 f_4| \lesssim \epsilon M t^\frac16 t^\frac14 \xx^{-\frac34} \cdot \epsilon^2 M^2 t^{-\frac12} \xx^{-\frac12}
= \epsilon^3 M^3  t^{-\frac1{12}} \xx^{-\frac54} ,
\]
which yields
\[
\| f_4 \|_{\dot H^{\frac12}} \lesssim \epsilon^3 M^3 t^{-\frac13},
\]
which suffices.

\bigskip

\textbf{(v) The bound for $f_5$.} 
For $\partial_x w_1$ and $\partial_x^2 w_1$ we have from \eqref{upm} and 
\eqref{point-boot-pm}:
\[
|\partial_x w_1| \lesssim M^2 \epsilon^2 
t^\frac12 \xx^{-\frac32}, \qquad  |\partial_x w_1| \lesssim M^2 \epsilon^2 
 \xx^{-1}.
\]
Hence, for $u w_{1,x}$ we get
\[
|u w_{1,x}|  \lesssim \epsilon^3 M^3
t^\frac14 \xx^{-\frac74}, \qquad  |\partial_x (u w_{1,x})| \lesssim \epsilon^3 
M^3 t^{-\frac14} \xx^{-\frac54}.
\]
Thus
\[
\left| |D|^\frac12 (u w_{1,x})\right| \lesssim \epsilon^3 M^3 \xx^{-\frac32}.
\]
and
\[
\| u w_{1,x}\|_{\dot H^\frac12} \lesssim \epsilon^3 M^3 t^{-\frac13}
\]
as needed.

The proof of Proposition~\ref{p:u2-dec} is concluded.

\newsection{Klainerman-Sobolev estimates}\label{s:KS}Our aim here is to prove the nonlinear Klainerman-Sobolev estimates 
in Proposition~\ref{p:KS}. We follow the spirit of the proof of Proposition~\ref{p:lin}, but with nonlinear adjustments.
Now the Sobolev bounds on $u$ and $L^{\NL} u$ have an $\epsilon$ factor, which we seek to recover linearly in the output. We can still use the scaling associated to the KdV equation to reduce the problem  to the case $t = 1$, following the setup in Section~\ref{ss:ell}.

Thus we are now working with the equation \eqref{resc-L}, which we recall here
\begin{equation}\label{resc-L-re}
(x - 3 \partial^2_x) \tilde u + 3  \tilde u^2 = \tilde f .
\end{equation}
Our bounds for $\tilde u$ and $\tilde f = L^{\NL}_{|t=1} u$ 
are now (see \eqref{resc-u} and \eqref{resc-f})
\begin{equation}\label{resc-u-re}
\| \tilde u \|_{B^{-\frac12}_{2,\infty}} \lesssim \tilde \epsilon ,
\end{equation}
respectively 
\begin{equation}\label{resc-f-re}
\| \tilde f \|_{\dot H^\frac12} \lesssim \tilde \epsilon,
\end{equation}
where 
\[
\tilde \epsilon = \epsilon t^{\frac13} \ll 1.
\]
 Finally, our bootstrap assumption \eqref{point-boot} on $u$ now reads
\begin{equation}\label{resc-boot-re}
| \tilde u(x)| \leq M \tilde \epsilon \xx^{-\frac14}, \qquad |\tilde u_x| \leq M \tilde \epsilon
\xx^\frac14.
\end{equation}

Here we can freely assume that $M \tilde \epsilon \ll 1$. Our goal will be to improve this by eliminating the constant $M$, and show that
\begin{equation}\label{resc-get}
| \tilde u(x)| \lesssim \tilde \epsilon \xx^{-\frac14}, \qquad |\tilde u_x| \lesssim \tilde \epsilon
\xx^\frac14.
\end{equation}
To keep the notations simple we will drop the tilde notation in what follows.

We note that the nonlinear part of $L^{\NL}$ is nonperturbative in this argument; however it is also nonresonant, which saves the day.

We will reuse here the results of Section~\ref{ss:ell} where we set $M_L = 1$. By Lemma~\ref{Lu-lo}
we have the low frequency bound
\begin{equation}\label{Lu-lo-re}
\| L^{\NL} u\|_{L^2(A_R)} \lesssim \epsilon R^\frac12,
\end{equation}
and by Lemma~\ref{l:u-hi} we have the high frequency bound
\begin{equation}\label{u-hi-re}
\|u\|_{L^2(A_R)} \lesssim \epsilon R^\frac14, \qquad
\|u_x\|_{L^2(A_R)} \lesssim \epsilon R^\frac34, \qquad \|u_{xx}\|_{L^2(A_R)} \lesssim \epsilon R^\frac54.
\end{equation}
Recall that here, due to the discussion in Section~\ref{ss:ell}, we can freely set $t=1$, and indeed arrive at the bounds above.

Also following the discussion in Section~\ref{ss:ell}, we can localize the problem to dyadic regions $\{ |x| \approx R \}$ where $R \gtrsim 1$.
Setting  $v := \chi_R u$, it follows that $v$ solves 
the equation
\begin{equation}
\label{new-eq}
(x -3  \partial_x^2 ) v + 3 uv = f,
\end{equation}
where $v$ and $f$ satisfy the bounds
\begin{equation}\label{est-v+}
\|v\|_{L^2(A_R)} \lesssim \epsilon R^{\frac14}, \qquad
\|v_x\|_{L^2(A_R)} \lesssim \epsilon R^\frac34, \qquad \|v_{xx}\|_{L^2(A_R)} \lesssim \epsilon 
R^\frac54,
\end{equation}
respectively 
\begin{equation}\label{est-u+}
\| f\|_{\dot H^\frac12} \lesssim \epsilon, \qquad \| f\|_{L^2} \lesssim \epsilon R^\frac12 .
\end{equation}
We now consider separately the three regions:

\medskip

\textbf{A. Pointwise estimate in the hyperbolic region: $-x \approx R \gg 1$.}
Here we consider the region $A_R^H$ to the left (of the origin), and use 
hyperbolic energy estimates to establish the desired pointwise 
bound for $u$ supported in $A^H_R$.
As in the linear argument, we consider an energy conservation type 
relation
\[
\frac{d}{dx} \left(-x |v|^2 + 3|v_x|^2  -2f_{\lo} v \right)= -|v|^2 - 2  f_{\hi} v_x -2  f_{lo,x} v -6 u v v_x.
\]
The nonlinear term is written in the form
\[
u v v_x = \frac13 \partial_x( \chi_R^2 u^3) +\frac{1}{3}  \chi_R \chi'_R u^3 
\]
The first term is added to the energy (this represents in this case a rudimentary normal form energy correction), so we get
\[
\frac{d}{dx} \left(- x |v|^2 + 3 |v_x|^2  -2 f_{\lo} v  -2 u v^2 \right)
= -|v|^2 - 2  f_{\hi} v_x -2   f_{lo,x} v  -  2\chi_R \chi'_R u^3 .
\]
Then applying Gronwall's inequality as in the linear case we obtain
\[
\sup_{x\in A^H_R}  \left\{ -x|v|^2 + 3 |v_x|^2\right\} \lesssim \epsilon^2 R^\frac12 + \sup_{x\in A^H_R} f_{\lo} v + R^{-1} \int \chi_R |u|^3\, dx .
\]
On the left the cubic correction $-uv^2$ is dominated by the main term $-x|v|^2$.
The second term on the right is as in the linear case, while for the last one we use the bootstrap assumption  to estimate
\[
R^{-1} \int \chi_R |u|^3\, dx  \lesssim M^3 \epsilon^3 R^{-\frac34},
\]
which is much better than needed.

\textbf{E. Pointwise estimate in the self-similar region $|x| \lesssim R = 1$}

Here we simply use Sobolev embeddings 
starting from the $u$ bounds in \eqref{u-hi-re}.

\medskip 

\textbf{F. Pointwise estimate in the elliptic region.}

Here we argue as in the proof of the linear estimate. The 
only difference is the nonlinear term $u^2$ in $L^{\NL} u$. In the 
hyperbolic region this term was nonperturbative but nonresonant.
Here the situation is simpler, as the nonlinear term is perturbative. Indeed in \eqref{new-eq} we can include the $3u$ coefficient with $x$. The $3u$ coefficient is negligible due to our bootstrap assumption \eqref{point-boot}. There we can proceed as in \eqref{bounds needed} in step F of the proof of 
Proposition~\ref{p:lin+}.

\medskip

\newsection{Solitons and inverse scattering}
The Lax operator associated to a state $u$ for the KdV equation
has the form
\[
L_u := -\partial_x^2 + u.
\]
The Lax pair associated to \eqref{kdv} is given by $L_u$ and $M$, where 
\[
M:=-4\partial_x^3 +6u\partial_x +3u_x,
\]
such that as for $u$ solving \eqref{kdv} we have
\[
\frac{d}{dt}L_u =\left[ M, L_u\right].
\]
This relation insures that the operators $L_u$ are unitarily equivalent 
in $L^2$ as $u$ evolves along the KdV flow. 

The inverse scattering theory, see \cite{ablowitz}, predicts that each state can be viewed as  a nonlinear superposition of solitons and dispersive states, where the solitons are associated to the eigenvalues of $L_u$.  As an example, the state 
\[
Q = 2 \sech^2 x 
\]
is a soliton which moves to the right with speed $4$, for which the corresponding Lax operator $L_Q$ has a single negative eigenvalue $\lambda = -1 $ with the corresponding  eigenfunction
\[
\phi = \sech x.
\]

Rescaling, we obtain the  soliton state 
\[
Q_\mu(x) = \mu^2 Q(\mu x),
\]
which moves to the right with speed $4 \mu^2$, for which the Lax operator 
has the eigenvalue $\lambda = -\mu^2$ and eigenfunction $\phi(\mu x)$.

More generally, if the Lax operator $L_u$ for a state $u$ has a negative  eigenvalue $-\mu^2$, then its evolution contains a soliton $Q_\mu$ which is 
localized to the spatial scale $\mu^{-1}$. 

For localized data, such a soliton would emerge from the dispersive wave at the 
time where the soliton scale matches the self-similar scale,
\[
\mu^{-1} = t^{\frac13}.
\]
In particular, for our $\epsilon$ size data, the cubic timescale corresponds
to $t = \epsilon^{-3}$ and thus to $\mu = \epsilon$. To see that solitons 
can only emerge at cubic time, and that this indeed happens, we will prove 
the following:

\begin{proposition}\label{spectrum} 
a) Assume that $u$ satisfies the smallness assumption \eqref{small}. Then 
any negative eigenvalue $\lambda_0$ for $L_u$  satisfies
\begin{equation}
- \lambda_0 \lesssim  \epsilon^{2}.
\end{equation}

b) Suppose that    $ \varepsilon \to u(\varepsilon)$ satisfies\footnote{This in particular guarantees that $\int_{\mathbb{R}} u \, dx $ is well defined.}
\begin{equation}\label{strongsmallness}   \Vert u \Vert_{\dot B^{-\frac12}_{2,\infty}} + \Vert x u \Vert_{\dot H^{\frac12}} < \varepsilon, \end{equation}  
as well as
\[
\lim_{\varepsilon \to 0} -\frac{1}{\varepsilon}\int_{\mathbb{R}} u(\varepsilon)\,  dx  =\ell  > 0.  
\] 
Then there exists $\varepsilon_0$ so that for $0<\varepsilon < \varepsilon_0$ there exists a smallest eigenvalue $\lambda(\varepsilon)$ and 
\[ -\lim_{\varepsilon\to 0 } \lambda(\varepsilon)/\varepsilon^2 = \frac14\ell^2.   \] 

c) If the negative part $u_-$ of $u$ satisfies
\[ \Vert x u_- \Vert_{L^1} \le N, \] 
then $L_u$ has at most N negative eigenvalues. 
 
d) Given $N\ge 1$, $\varepsilon>0$ there exists $u \leq 0$ such that 
\[ \Vert xu \Vert_{L^1} \le N-1+\varepsilon \] 
and $L_u$ has $N$ negative eigenvalues.
 \end{proposition}

\begin{remark} For part (a) we only use the small Besov norm, without the decay in the second term in \eqref{strongsmallness}. But even adding this decay, it is still possible   
have infinitely many negative eigenvalues. The parts (c) and (d) of the above proposition clarify the additional decay which would 
be needed in order to have finitely many eigenvalues. Part c)  has been proven by Seto \cite{MR0340846}. We provide a short elementary argument. 
\end{remark}
\begin{proof}
a) To show the lowest eigenvalue is at least $- \epsilon^2$ we need the inequality
\[
\int_{\mathbb{R}} u \phi^2 dx \lesssim \| \nabla \phi \|_{L^2}^2 + \epsilon^2 \| \phi\|_{L^2}^2.
\]
Here we use only the Besov norm $B^{-\frac12,\epsilon}_{2,\infty}$ for $u$
which guarantees that 
\[
\| u_{\leq \epsilon}\|_{L^\infty} \lesssim \epsilon,
\]
and
\[
\| u_\lambda \|_{L^2} \lesssim \epsilon \lambda^\frac12, \qquad \| u_\lambda \|_{L^\infty} \lesssim \epsilon \lambda,
\qquad \lambda \geq \epsilon.
\]
Now we use the Littlewood-Paley trichotomy to estimate the left hand side,
\[
\int_{\mathbb{R}} u \phi^2 \, dx \lesssim \sum_{\epsilon \leq \lambda,\lambda_1,\lambda_2}\int_{\mathbb{R}} u \phi^2 \, dx .
\]
If $\lambda < \lambda_1 = \lambda_2$ we use the $L^\infty$ bound for $u_\lambda$ to get
\[
\sum_{\lambda_1 > \epsilon} \epsilon \lambda_1 \| u_{\lambda_1}\|_{L^2}^2,
\]
which is controlled by the right hand side.

On the other hand if $\lambda = \lambda_1 > \lambda_2$ then we use $L^\infty$ for $u_{\lambda_2}$
to get the bound
\[
\sum_{\lambda_1 > \lambda_2 > \epsilon} \epsilon \lambda_1^\frac12 \| u_{\lambda_1}\|_{L^2}\,
 \lambda_2^\frac12 \| u_{\lambda_2}\|_{L^2},
\]
which is again estimated by the right hand side.

\bigskip

b) We observe that 
\[  \varepsilon^{-2}   u(x/\varepsilon) \to - \ell \delta_0 \] 
 in $H^{-1}+ L^\infty$. On the other hand the eigenvalues depend continuously on the potential in $H^{-1}+ L^\infty$. But is not hard to check that the potential $-\ell \delta_0$ 
 yields exactly the simple eigenvalue $-(\ell/2)^2$.

\bigskip

c) Replacing $u$ by $-u_-$  decreases the eigenvalues 
of $L_u$, so without any restriction in generality 
we can assume that $u \leq 0$.

Suppose that there are at least $N+1$ nonpositive eigenvalues. Then the $N+1$-th eigenfunction $\phi$ has $N$ points of vanishing, and $N+1$ nodal intervals. Let $(x_0,x_1)$ be one of them. 

The operator 
$L_u$ restricted to $[x_0,x_1]$ with Dirichlet boundary condition
has at least one negative eigenvalue, with the
 restriction of $\phi$ as the corresponding eigenfunction.
On the other hand $L_0$ with the same Dirichlet boundary condition
is positive. Hence a continuity argument shows that there 
exists an unique $h \in (0,1)$ so that the operator $L_{hu}$
has $0$ as the lowest eigenvalue. We denote by $\psi$ the 
corresponding eigenfunction, solving
\[ 
\psi(x_0)=\psi(x_1)= 0,\quad  \quad  -\psi'' + h u  \psi = 0.
\] 
We can freely assume that $\psi > 0$ in $[x_0,x_1]$.
Then $\psi$ is concave there, so we can also assume that 
$\psi'(x_0) = 1$. Hence $\psi(x) \leq x - x_0$ and $\psi'(x_2) < 0$. Thus 
\[
1 < \psi'(x_0) - \psi'(x_1) = -\int_{x_0}^{x_1} h u(t) \psi(x)
\, dx \leq -\int_{x_0}^{x_1} u(t) (t-x_0) 
\, dx
\] 
with equality iff $u$ is a Dirac measure and $h = 1$. 
If $(x_0,x_1)$ is a nodal interval with $0 \le x_0$ then 
\[ 
1 < \int_{x_0}^{x_1} u(t) (t-x_0)\, dx \le \int_{x_0}^{x_1} t u(t) \, dt. 
\]
The argument for $x_1= \infty$ is similar. If on the other hand $x_1 < 0$ then we interchange the roles of $x_0$ and $x_1$ and the same conclusion follows.
Since there are  $N$ nodal intervals not containing $0$, it follows 
that 
\[
\int |x| |u(x)| \, dx > N,
\]
which yields a contradiction.

\bigskip

d) Suppose we find $N+1$ points
\[ 
-\infty < x_0 < \dots < x_{N} < \infty, 
\]
a measure $u$ in $(x_0,x_{N})$
and a solution $\phi$ to 
\[ 
- \phi'' + u \phi = 0 \qquad \text{ in } (x_0,x_{N})
\] 
vanishing at these $N+1$ points.

Then $\phi$ is an eigenfunction to the $N$-th eigenvalue of the Schrödinger operator on $(x_0,x_N)$ with Dirichlet boundary condition. By the variational characterization of eigenfunctions we see that the Schrödinger operator on $\R$ has at least $N$ negative eigenvalues.  
We construct a sum of Dirac measures and $\phi$ with these properties. A simple approximation argument yields the full result. 

We choose $x_0=-1$, and a sequence of points $x_0=-1 < 0=y_1 < x_1 < y_2 < \dots < y_{N}< x_{N} $ and  we put the Dirac masses at the points $y_j$. We choose $\phi$ continuous and affine on $[x_0,y_1]$, $[y_j,y_j{j+1}]$ and $[y_{N},x_{N}]$. Let $\Delta \phi(y_j) $ be the jump of the derivatives at this point. We assume $y_j$ to be a point of a local maximum of $|\phi|$. Then 
 \[  
 \phi''(y_j) = \Delta \phi(y_j) \delta_{y_j}= \frac{\Delta \phi(y_j)}{\phi(y_j)} \phi(y_j) \delta_{y_j} .
 \] 
 Starting at $x_0=-1$, $y_1=0$, $x_1=1$ and 
 \[ 
 \phi = x-1 \qquad \text{ for } -1\le x \le 0, \qquad \phi(x)=-1+x  \quad \text{ for } 0\le x \le y_2  
 \] 
 with $y_2$ to be chosen. We put a multiple of a Dirac measure at $y_2$
\[
- \phi''-  \frac{1+\varepsilon/(2N)}{y_2-1} \delta_{y_2} \phi =0 .
\]
 Then $\phi(y_2)=1-y_1$, $\phi'(y_2+) > 0$. If we choose $y_2$ large we can ensure that its contribution to the $L^1$ norm of $xu$ is only slightly larger than 1,
\[  
y_2 \frac{1+\varepsilon/(2N)}{y_2-1}  < 1+\varepsilon/N. 
\]
After the point $y_2$ the function $\phi$ is linearly increasing. We denote by $x_2$ the point where it vanishes and then repeat the procedure to chose $y_2 <x_2<y_3$, as the location of the next Dirac mass in $u$.  We repeat this procedure  to construct all the $y$'s and $x$'s. 
\end{proof} 

The inverse scattering method allows to study solutions under stronger conditions as in this paper, but for all times. This is a nontrivial task. Here we adapt and explain results of Schuur \cite{schuur} for special initial data. We fix a Schwartz function $\phi_0$ with 
\[ 
\int_{\mathbb{R}} \phi_0\,  dx = -1. 
\]
We consider the initial data $u_0 = \varepsilon\phi_0$. It satisfies the smallness condition if $\varepsilon>0$ is sufficiently small. 
By Proposition \eqref{spectrum} we know that there is exactly one negative eigenvalue $-\lambda$ of size $-\varepsilon^2$. The corresponding pure soliton is 
\[
- 2\lambda \sech^2( \sqrt{\lambda}(x-y_0-\lambda  t )).
\] 
Schuur proved that there exists $y_0$ with $|y_0| \lesssim 1 $ so that to the right of the self-similar region we get 
\[ 
\Vert u(t)+ 2 \lambda \sech^2(\sqrt{\lambda}(x-y_0-\lambda t)) \Vert_{L_x^\infty(-c_1 t^{\frac13},\infty)} \le c_2 t^{-\frac13},
\] 
for all $t \ge t_0$, with precise formulas for the constants $t_0$, $y_0$, $c_1$ and $c_2$. It is not too hard to check their size: 
\[ 
1 \lesssim  c_1, \qquad |y_0| \lesssim 1, \qquad c_2 \lesssim \varepsilon, \qquad t_0 \ge \varepsilon^{-3}. 
\]
If $t \sim \varepsilon^{-3} $ the size of the soliton is the same as the size of the error estimate, and this is the scale on which the soliton emerges. 




\end{document}